\documentclass[a4paper]{amsart}

\usepackage{amscd,amsthm,amsfonts,amssymb,amsmath,esint}
\usepackage[colorlinks=true]{hyperref}
\usepackage{fullpage}
\usepackage[all]{xy}

\usepackage{mathrsfs}
\usepackage{tikz-cd}

\newcommand{\msb}{\mathscr{B}}

\newcommand{\mst}{\mathscr{T}}

    \newcommand{\BC}{{\mathbb {C}}} 
     \newcommand{\BF}{{\mathbb {F}}}

    \newcommand{\BQ}{{\mathbb {Q}}} \newcommand{\BR}{{\mathbb {R}}}

     \newcommand{\BZ}{{\mathbb {Z}}}

     \renewcommand{\CD}{{\mathcal {D}}}
     \newcommand{\CF}{{\mathcal {F}}}
    \newcommand{\CG}{{\mathcal {G}}} 
     \newcommand{\CJ}{{\mathcal {J}}}
    \newcommand{\CK}{{\mathcal {K}}} 
     
    \newcommand{\CO}{{\mathcal {O}}} 
     \newcommand{\CR}{{\mathcal {R}}}

    \newcommand{\fa}{{\mathfrak{a}}} \newcommand{\fb}{{\mathfrak{b}}}
    \newcommand{\fc}{{\mathfrak{c}}} \newcommand{\fd}{{\mathfrak{d}}}
     
    \newcommand{\fg}{{\mathfrak{g}}}

    \newcommand{\fm}{{\mathfrak{m}}} 
     \newcommand{\fp}{{\mathfrak{p}}}
    \newcommand{\fq}{{\mathfrak{q}}}

    \newcommand{\fC}{{\mathfrak{C}}} \newcommand{\fD}{{\mathfrak{D}}}
     \newcommand{\fF}{{\mathfrak{F}}}

     \newcommand{\fP}{{\mathfrak{P}}}

     \newcommand{\fX}{{\mathfrak{X}}}

    \newcommand{\Cl}{{\mathrm{Cl}}}
    \newcommand{\Coker}{{\mathrm{Coker}}}
    \newcommand{\cyc}{{\mathrm{cyc}}}

    \newcommand{\End}{{\mathrm{End}}}

    \newcommand{\Gal}{{\mathrm{Gal}}} 
    
    \newcommand{\Hom}{{\mathrm{Hom}}}

    \newcommand{\Ker}{{\mathrm{Ker}}}

    \newcommand{\ord}{{\mathrm{ord}}}

    \renewcommand{\mod}{\ \mathrm{mod}\ }

    \newcommand{\Sel}{{\mathrm{Sel}}}

    \font\cyr=wncyr10

    \newcommand{\Sha}{\hbox{\cyr X}}

    \newcommand{\ov}{\overline}

    \theoremstyle{plain}
    \newtheorem{thm}{Theorem}[section] \newtheorem{cor}[thm]{Corollary}
    \newtheorem{lem}[thm]{Lemma}  \newtheorem{prop}[thm]{Proposition}

\theoremstyle{remark} 
\theoremstyle{remark} 
\theoremstyle{remark} 

    \numberwithin{equation}{section}

\newcommand{\hilbert}[3]{\Bigl(\dfrac{#1,#2}{#3}\Bigr)} 

\begin{document}

\title {Classical Iwasawa theory and infinite descent on a family of abelian varieties}

\author{John Coates, Jianing Li,  Yongxiong Li}



\begin{abstract} For primes $q \equiv 7 \mod 16$, the present manuscript shows that elementary methods enable one to prove
surprisingly strong results about the Iwasawa theory of the Gross family of elliptic curves with complex multiplication by the ring of integers
of the field $K = \BQ(\sqrt{-q})$, which are in perfect accord with the predictions of the conjecture of Birch and Swinnerton-Dyer.
We also prove some interesting phenomena related to a classical conjecture of Greenberg, and give a new proof of an old theorem of Hasse.
\end{abstract}

\maketitle

\section{Introduction} While there has been great progress on the analytic side of Iwasawa theory over the last forty years, little progress has been made
on some of the important concrete classical questions of Iwasawa theory and their connexion with descent theory on abelian varieties. The aim of the present 
paper is to establish some new results in this direction by surprisingly elementary methods. Throughout, $q$ will denote a prime such that $q \equiv 7 \mod 8$, and we let
$K = \BQ(\sqrt{-q})$. Write $\CO_K$ for the ring of integers of $K$, and $h$ for the class number of $K$. The prime 2 splits in $K$, and we put $2\CO_K = \fp\fp^*$. By class field theory, $K$ has a unique $\BZ_2$-extension $K_\infty$ (resp. $K_\infty^*$) which is unramified outside $\fp$ (resp. $\fp^*$), and $\fp$ (resp. $\fp^*$) is then totally ramified in $K_\infty$ (resp. $K_\infty^*$), since $h$ is odd. Moreover, we let $\CK_\infty = K_\infty K_\infty^*$ be the composite of all $\BZ_2$-extensions of $K$. If $R$ is any algebraic extension of $K$, we define $L(R)$ to be the maximal unramified abelian $2$-extension of $R$, and $M(R)$ (resp. $M^*(R)$) to be the maximal abelian $2$-extension of $R$, which is unramified outside the primes of $R$ lying above $\fp$ (resp. above $\fp^*$). We also define
\begin{equation}\label{a2}
R_\infty = RK_\infty, \, \, R_\infty^* = RK_\infty^*, \, \, \CR_\infty = R\CK_\infty.
\end{equation}
Noting that obviously $M(R)$ (resp. $M^*(R)$) contains $R_\infty$ (resp. $R^*_\infty$), we define
\begin{equation}\label{a1}
X(R) = \Gal(M(R)/R_\infty), \, \, X^*(R) = \Gal(M^*(R)/R^*_\infty), \, \,  Y(R) = \Gal(L(R)/R).
\end{equation}
Consider now the polynomial
\begin{equation}\label{pol} 
f(x) = x^4 + q.
\end{equation}
Our starting point will be the following four extensions of $K$ arising from this polynomial. Let $\alpha$ be any fixed root of $f(x)$, and define
\begin{equation} \label{a3}
F = K(\alpha), \,  F'= K(\alpha \sqrt{-1}), \, D = K(\sqrt{-1}), \, J = K(\alpha, \sqrt{-1}).
\end{equation}
\noindent Thus $J$ is the splitting field of the polynomial \eqref{pol}, so that the primes $\fp, \fp^*$ have entirely similar decompositions in $J$. Note also that $D$ is Galois over $\BQ$, but that the fields $F$ and $F'$ are not Galois over $\BQ$, only being isomorphic as extensions of $\BQ$. By the definition \eqref{a2}, $X(J_\infty)=\Gal(M(J_\infty)/J_\infty)$ and $X^*(J^*_\infty)=\Gal(M^*(J^*_\infty)/J^*_\infty)$.

\begin{thm}\label{mir} For all primes $q \equiv 7 \mod 16$,  $X(J_\infty)$ and $X^*(J_\infty^*)$ are both free $\BZ_2$-modules of rank $1$.
\end{thm}

\noindent  This result has an interesting connexion with the arithmetic of a certain abelian variety with complex multiplication. Let $H = K(j(\CO_K))$ be the Hilbert class field of $K$; here $j$ is the classical modular function on lattices in the complex plane. Gross \cite{Gross1} showed that  that there is an interesting elliptic curve $A$ defined over $\BQ(j(\CO_K))$, with complex multiplication by $\CO_K$, minimal discriminant $(-q^3)$, and which is a $\BQ$-curve in the sense that it is isogenous to all of its conjugates. Let $B/K$ be the $h$-dimensional abelian variety which is the restriction of scalars of $A$ from $H$ to $K$. We refer the reader to \cite{Gross1} and \cite{CL} for a more detailed discussion of the following basic facts about $B$. Let $\msb = \End_K(B)$, and  $\mst = \msb \otimes \BQ$, so that $\mst$ is a CM field of degree $h$ over $K$. Then $\msb$ is an order in $\mst$, which is ramified over $\CO_K$ at precisely the primes dividing $h$ (see \cite{Gross1}, Theorem 15.2.5). In particular, since $h$ is odd, the primes $\fp$, and $\fp^*$ are both unramified in $\mst$.  Now the torsion subgroup of $B(K)$ is 
$\CO_K/2\CO_K$, and the action of $\msb$ on this torsion subgroup  gives an $\CO_K$-algebra surjection from $\msb$ onto $\CO_K/2\CO_K$, whose kernel is the product of two conjugate primes $\fP, \fP^*$ of $\msb$ lying above $\fp, \fp^*$, respectively. These primes  are both unramified in $\msb$,  and have residue fields equal to the field with $2$ elements $\BF_2$.  Now we recall that for any algebraic extension $\fF$ of $K$, the Tate-Shafarevich group of $B/\fF$ is defined by
\begin{equation}\label{a4}
\Sha(B/\fF) = \Ker (H^1(\fF, B) \to \prod_{v}H^1(\fF_v, B)),
\end{equation}
where $v$ runs over all non-archimedean places of $\fF$, and, as usual, $\fF_v$ denotes the union of the completions at $v$ of the finite extensions of $K$ contained in $\fF$. Thus $\Sha(B/\fF)$ is a torsion $\msb$-module, and we write
$\Sha(B/\fF)(\fP)$ for its $\fP$-primary subgroup. Let us also mention that $B(\fF)\otimes_\BZ \BQ$ has a natural action by the field $\mst$, and thus its $\BQ$-dimension, if finite,will always be a multiple of $2h$. Now, as we shall explain in $\S{4}$, we have
\begin{equation}\label{a5}
J_\infty = K(\sqrt{-1}, B_{\fP^\infty}),  \,   \, \, J_\infty^* = K(\sqrt{-1}, B_{\fP^{*\infty}});
\end{equation}
here $B_{\fP^\infty}$ (resp. $B_{\fP^{*\infty}}$) denotes the Galois module of $\fP^\infty$ (resp. $\fP^{*\infty}$)-division points in $B(\ov{K})$, where $\ov{K}$ is the algebraic closure of $K$.  Then, for primes $q \equiv 7 \mod 16$, an equivalent form of Theorem \ref{mir} in terms of the abelian variety $B$ is as follows (see Theorem \ref{f10}). We have that either $B(J_\infty)\otimes_\BZ\BQ $ has $\BQ$-dimension $2h$ and $\Sha(B/J_\infty)(\fP) = 0$, or $B(J_\infty)\otimes_\BZ\BQ = 0$ and  $\Sha(B/J_\infty)(\fP) = \BQ_2/\BZ_2$. An entirely parallel result holds for $B/J^*_\infty$, but with the $\fP^*$-primary subgroup of $\Sha(B/J^*_\infty)$. Note also that, since the abelian variety $B$ is defined over $\BQ$, we have $B(J_\infty)\otimes_\BZ\BQ \cong B(J^*_\infty)\otimes_\BZ\BQ$ as $\BQ$-vector spaces.  In fact, we show in \S 4, by a further purely arithmetic argument, that one can establish the following stronger result. Recall that $D = K(\sqrt{-1})$, so that $J_\infty/D_\infty$ and $J^*_\infty/D^*_\infty$ are both quadratic extensions. 

\begin{thm}\label{m3} Assume $q \equiv 7 \mod 16$. Then precisely one of the two following options is valid:- (i) both $B(D_\infty)\otimes_\BZ\BQ $ and $B(D^*_\infty)\otimes_\BZ\BQ $ have $\BQ$-dimension $2h$, and $\Sha(B/D_\infty)(\fP) = \Sha(B/D^*_\infty)(\fP^*) = 0$, or (ii) $B(D_\infty)\otimes_\BZ\BQ = B(D^*_\infty)\otimes_\BZ\BQ = 0$, and  both $\Sha(B/D_\infty)(\fP)$ and $\Sha(B/D^*_\infty)(\fP^*)$ are isomorphic to $\BQ_2/\BZ_2$ as abelian groups. 
\end{thm}

However, using the complex $L$-series of $B/D$ and the theory of Heegner points as in \cite{MY} , we will point out at the end of \S4 that one can in fact obtain the following result from these analytic arguments.

\begin{thm}\label{mirB} For all primes $q \equiv 7 \mod 8$, $B(D)\otimes_\BZ \BQ$ has dimension $2h$ as a $\BQ$-vector space,  and  $\Sha(B/D)$ is finite.
\end{thm}

\noindent Combing this result with Theorem \ref{m3}, it follows that, for all primes $q \equiv 7 \mod 16$,  we have $B(D)\otimes_\BZ \BQ \cong B(D_\infty)\otimes_\BZ \BQ \cong B(J_\infty)\otimes_\BZ \BQ$ are all $\BQ$-vector spaces of dimension $2h$, and  also $\Sha(B/D_\infty)(\fP) = \Sha(B/J_\infty)(\fP) = 0$; an entirely similar assertion is valid
for the fields $D^*_\infty$ and $J^*_\infty$, with $\fP^*$ replacing $\fP$.

\bigskip

As a second application of our method, we shall establish the following result. We remark that, for all primes $q \equiv 3 \mod 4$, the class numbers of the fields $K, D, F, F', J$ are all known to be odd (see Proposition \ref{prop: clno1}).

\begin{thm}\label{thm: Y} For all primes $q \equiv 7 \mod 16$, we have $Y(\CK_\infty) = Y(\CD_\infty)=Y(\CF_\infty) = Y(\CF'_\infty) = Y(\CJ_\infty) = 0$, or equivalently none of the fields
$\CK_\infty, \CF_\infty, \CF'_\infty, \CJ_\infty$ has a non-trivial unramified abelian $2$-extension.
\end{thm}
\noindent We remark that the assertion $Y(\CK_\infty) = 0$ proves an interesting but very special case of a classical conjecture of Greenberg, \cite{RG1} Conjecture 3.5,  and the remaining assertions of the theorem provide interesting examples of the mystery as to why certain arithmetic Iwasawa modules are smaller than one would expect from the basic classical theory.

\medskip

In all that follows, we shall always adopt the following notational convention. By Lemma \ref{lem: ram F/K}, exactly one of the prime factors $\fp$ and $\fp^*$ of $2$ in $K$ will ramify in the field $F$, and we shall always denote this factor by $\fp$. The other prime factor $\fp^*$ will then denote the unique prime factor of $2$ which ramifies in the field $F'$.
For brevity, in what follows we shall sometimes only state a result for the field $F$, and leave it to the reader to formulate the totally analogous result for the field $F'$,
but with the primes $\fp$ and $\fp^*$ interchanged.
\medskip

The proof of Theorem \ref{mir} makes crucial use of a weaker form of a second arithmetic phenomena related to the fields $F$ and $F'$, which, for simplicity, we only state for $F$. The group of global units of $F$ has rank $1$, and we write $\eta$ for any generator of this group modulo torsion.
Let $v$ be any of the primes of $F$ lying above $2$. Then $\log_v(\eta)$ is well defined up to a sign; here we take the usual extension of the $v$-adic logarithm to the group of units of the ring of integers of the completion  $F_v $ of $F$ at $v$. As usual, $\ord_v$ will denote the additive valuation on $F_v$, normalized to that the order of a local parameter at $v$ is 1.

\begin{thm}\label{thm: log} Assume $q \equiv 7 \mod 16$.  Then $\fp$ is ramified in $F$ with ramification index $2$, and $\fp^*$ is unramified with inertial degree 2 in $F$. We write $w$ (resp. $w^*$) for the unique prime of $F$ above $\fp$ (resp. $\fp^*$). Then, if $\eta$ is a fundamental unit of $F$, we have
\begin{equation}\label{ord}
\ord_w(\log_w(\eta)) = 2, \, \, \ord_{w^*}(\log_{w^*}(\eta)) =  3.
\end{equation}
\end{thm}
\noindent  In fact, always assuming that $q \equiv 7 \mod 16$, the first assertion is already shown to hold in \cite{CL} as a consequence of the Iwasawa theory at the prime $p=2$ of the abelian variety $B/K$, and a second very short and elementary proof is given in \cite{L1}.

\section{Proof of Theorem~\ref{mir}}
It follows from class field theory that $M(K)=K_\infty$, whence, since the $\Gamma$-coinvariants $X(K_\infty)_\Gamma$ of $X(K_\infty)$ is $X(K)$, it follows from Nakayama's lemma that  $X(K_\infty) = 0$ for any prime $q\equiv 7\bmod 8$; here $\Gamma = \Gal(K_\infty/K)$ and $\Gamma$ acts on $X(K_\infty)$ as usual by lifting inner automorphisms.
The goal of this section is to compute $X(R)$ for $R= F, F', D, J, F_\infty, F'_\infty, D_\infty, J_\infty$. The analogous results for the Galois groups $X^*(R)$ then follow from the following lemma. 
\begin{lem}\label{lem: symmetry}	
Assume $q\equiv 7\bmod 8$. Then, as $\BZ_2$-modules, we have $X(J_\infty)\cong X^*(J^*_\infty)$, $X(D_\infty)\cong X^*(D^*_\infty)$, $X(F_\infty)\cong X^*(F'^*_\infty)$, $X(F'_\infty) \cong X^*(F^*_\infty)$, $X(F)\cong X^*(F')$, and $X(F') \cong X^*(F)$.  
\end{lem}

\begin{proof}
We fix an algebraic closure $\ov{\BQ}$ of $\BQ$, and all of our relevant fields will be assumed to lie in it. 	
Let $\sigma$ be any element
of $\Hom(M(J_\infty), \ov{\BQ})$ such that the restriction of $\sigma$ to $K$ is the non-trivial element of $\Gal(K/\BQ)$. Now $\sigma(M(J_\infty))$ is an abelian extension of $\sigma(J_\infty)$, which is unramified outside the primes above $\sigma(\fp)=\fp^*$.
 Now $\sigma(J) = J$ because $J$ is Galois over $\BQ$,
and thus $\sigma(J_\infty)$ will be a $\BZ_2$-extension of $J$, which is unramified outside $\sigma(\fp) = \fp^*$. Hence $\sigma(J_\infty) = J_\infty^*$. Thus we have an embedding $ M(J_\infty) \hookrightarrow M^*(J^*_\infty) $ and similarly $M^*(J^*_\infty) \hookrightarrow M(J_\infty)$. This implies $\sigma(M(J_\infty))=M^*(J^*_\infty)$ whence $\sigma$ induces an isomorphism $X(J_\infty)\cong X^*(J^*_\infty)$. Since we have $\sigma( F ) = F'$, $\sigma(F_\infty) = F'^*_\infty$, $\sigma(D)=D$, $\sigma(D_\infty)=D^*_\infty$, the remaining isomorphisms in Lemma~\ref{lem: symmetry} are also induced by $\sigma$, and they can be proved in the same way.
\end{proof}

\begin{prop}\label{prop: X(D)}
Assume $q\equiv 7\bmod 8$. Then
the modules $X(F_\infty), X^*(F'^*_\infty), X(D_\infty),$ and  $X^*(D^*_\infty)$ are all finitely generated free $\BZ_2$-modules. Moreover, these modules are all zero if $q\equiv 7\bmod 16$, and are all nonzero if $q\equiv 15\bmod 16$.
\end{prop}

\begin{proof}
In fact, it is already proven in 
\cite{CL}, Theorem 3.1 that, for all primes $q \equiv 7 \mod 8$,   $X(F_\infty)$ is a free finitely generated $\BZ_2$-module, and exactly the same arguments can be used to show that $X(F'_\infty)$ and $X(D_\infty)$ are also free finitely generated $\BZ_2$-modules. Also in \cite{CL}, it is proved that $X(F_\infty)=0$ for any $q\equiv 7\bmod 16$.  The elementary arguments used in \cite{L1} show that $X(F_\infty)$ is non-zero when $q \equiv 15 \mod 16$. By Lemma~\ref{lem: symmetry}, the same assertion holds for $X^*(F'^*_\infty)$. 

It remains to compute $X(D_\infty)$.  Noting that $X(D_\infty)_\Gamma = X(D)$, where $\Gamma = \Gal(D_\infty/D)$, it suffices by Nakayama's lemma to show that $X(D)$ is zero if $q\equiv 7\bmod 16$ and nonzero if $q\equiv 15\bmod 16$. Clearly, there is one prime $v$ of $D$ lying above $\fp$ which is ramified in $D/K$. The class number of $D$ is odd (see Proposition~\ref{prop: clno1}), the local discriminant ideal of $D/K$ at $\fp$ is generated by $4$ in the completion $K_\fp=\BQ_2$, and the group of roots of unity in $D$ is $\mu_4=\langle  i \rangle$ where $i=\sqrt{-1}$. Let $\xi$ be a fundamental unit of $D$.  By a classical formula \cite{CW}, whose proof we recall in the Appendix, our assertion about  $X(D)$ will follow if we can show that
\begin{equation}\label{eq: log unit D}
\ord_{\fp} (\log_v( \xi )) \begin{cases}
=2  & \text{ if } q\equiv 7\bmod 16,\\
> 2 & \text{ if } q\equiv 15\bmod 16.
\end{cases}
\end{equation} 
Here $\log_v$ is the $v$-adic logarithm, so $\log_v(\xi)\in D_v$; also note that $\log_v(\xi)$ converges since $\xi\equiv 1\bmod v$.
Now let $D^{+}=\BQ(\sqrt{q})$ be the maximal real subfield of $D$. 
Let $\CO^\times_D$ (resp. $\CO^\times_{D^+}$) be the group of units of $D$ (resp. $D^+$). We claim that the unit index 
\begin{equation}\label{eq: index CM}
[ \CO^\times_D: \mu_4 \CO^\times_{D^+} ]=2.
\end{equation}
Let $\sigma$ denote the nontrivial automorphism of the CM extension $D/D^+$. 
Then the unit $\xi^{1-\sigma}=\xi^\sigma/ \xi $ must be a root of unity of $D$, since all the conjugates of it have complex absolute value $1$. This gives rise to an exact sequence
\[\begin{tikzcd}
 1\longrightarrow \CO^\times_{D^+} \longrightarrow \CO^\times_D \xrightarrow{1-\sigma}   \mu_4. \end{tikzcd}\]
In particular,  we have $[\CO^\times_D:  \CO^\times_{D^+} ]\leq 4$ whence 
\begin{equation}\label{eq: index CM2}
[\CO^\times_D:   \mu_4 \CO^\times_{D^+}   ]\leq 2.\end{equation}
Let $\varepsilon \in \CO^\times_{D^+}$ be the fundamental unit of $D^+$. 
Since the class number of $D^+$ is odd as $D$ has odd class number, the prime of $D^+$ lying above $2$ which is ramified in $D^+/\BQ$ is a principal ideal $(\theta)$ of $D^+$. It follows that $\theta^2/2$ is a unit of $D^+$, and it must be an odd power of $\varepsilon$; otherwise $\sqrt{2}$ would belong to $D^+$ which is plainly impossible. In fact by a suitable choice of $\theta$, we may assume $\varepsilon =\theta^2/2$. 
Note that $\theta/(1+i)$ is a unit in $D$ and that $i(\theta/(1+i))^2 = \varepsilon$. This implies that $2$ divides the unit index of \eqref{eq: index CM} whence by \eqref{eq: index CM2} this index is indeed $2$ which proves the claimed equality \eqref{eq: index CM}. It follows that, if we write $\varepsilon=\xi^a i^b$ with $a,b\in \BZ$, we must have $2$ exactly divides $a$. Thus, in order to prove \eqref{eq: log unit D}, we need to show that 
\begin{equation}\label{eq: log unit D+}
\ord_{\fp} (\log_v( \varepsilon ))  \begin{cases}
=3  & \text{ if } q\equiv 7\bmod 16,\\
> 3 & \text{ if } q\equiv 15\bmod 16.
\end{cases}
\end{equation}
Write $\theta = x+y\sqrt{q}$ with $x,y \in \BZ$. Since $\theta$ is a generator of the prime of $D^+$ above $2$, it follows that,  as $q\equiv 7\bmod 8$, we have $x^2-y^2q=2$. Hence $x^2\equiv 2 \bmod y$. Thus every prime divisor of $y$ is congruent to $\pm 1$ modulo $8$, whence $y^2\equiv 1\bmod 16$ since $y$ is clearly odd. Then, writing $\mathrm{Tr}$ for the trace map from $D^+$ to $\BQ$, it follows that $\mathrm{Tr} (\varepsilon)= x^2+y^2 q =2 +2y^2 q \equiv 2(1+q) \bmod 32$.  We conclude that 
\begin{equation}\label{eq: ord trace} \ord_\fp( \mathrm{Tr}(\varepsilon)   )  \begin{cases} = 4 & \text{ if } q\equiv 7\bmod 16,\\
> 4  & \text{ if } q\equiv 15\bmod 16 \end{cases}. \end{equation}
Now since $q\equiv 7\bmod 8$, $-1$ is not a norm from $D^+$, whence $N_{D^+/\BQ}(\varepsilon)=1$. It follows that $\ord_\fp(\varepsilon^2+1)= \ord_\fp( \mathrm{Tr}(\varepsilon) )$. Thus $\ord_{\fp}(\varepsilon^2-1)= \ord_\fp(\varepsilon^2+1 -2) = 1$. This implies that $\ord_{\fp}(\varepsilon^4-1)=1+\ord_\fp( \mathrm{Tr}(\varepsilon)   )$. By a basic property of logarithm series, we must then have $\ord_{\fp} (\log_v( \varepsilon^4 ))=1+\ord_\fp( \mathrm{Tr}(\varepsilon)   )$. 
Now the assertion \eqref{eq: log unit D+} follows from \eqref{eq: ord trace}. From this, we obtain \eqref{eq: log unit D}. This completes the proof of Proposition~\ref{prop: X(D)}. \end{proof}


Let $K_n$ (resp. $K^*_n$) be the $n$-th layer of the $\BZ_2$-extension $K_\infty/K$ (resp. $K^*_\infty/K$). The following lemma, which has been proved in \cite[Lemma 3.3]{CL}, will be vital for our arguments, and we give an alternative proof here using the language of ideles.

\begin{lem}\label{lem: proc}
	Let $r_q= 2^{\ord_2(q+1)-3}$. Then there are precisely $r_q$ primes of $K_\infty$ lying above $\fq$. In particular, $\fq$ is inert in $K_\infty$ if $q\equiv 7\bmod 16$, and $\fq$ splits in the first layer of $K_\infty$ if $q\equiv 15\bmod 16$.
\end{lem}

\begin{proof}
In $K_\fp=\BQ_2$, we have $3\leq \ord_2(q+1)=\ord_2( \sqrt{-q}+1 )+\ord_2(\sqrt{-q}-1)$. We may assume that, $\ord_\fp(\sqrt{-q}-1 )>1$ whence $\ord_\fp(\sqrt{-q}+1)=1$. Put $k= \ord_\fp (\sqrt{-q}-1  )$. Then $\ord_2(q+1 )=k+1$ and $r_q = 2^{k-2}$.
By class field theory and the fact that $h$ is odd, the following composite map is an isomorphism 
\[
\begin{tikzcd}
U_\fp/\{\pm 1\}\ar[r, hook, "i_\fp" ] & I_K/\overline{K^\times \prod_{v\neq \fp}U_v}  \ar[r, twoheadrightarrow]  &  \Gal(M(K)/K).
\end{tikzcd} 
\]
Here $I_K$ is the idele group of $K$, $U_v$ denotes the group of local units at a finite place $v$ of $K$, and the last map is the Artin map. This shows $M(K)=K_\infty$, since the left hand side is isomorphic to $\BZ_2$. Of course, this result can also be proven by elementary Iwasawa theory, using the fact that $h$ is odd (see Lemma 3.2 of \cite{CL}). The above composite map sends $\sqrt{-q}$ to the inverse of the Frobenius of $\fq$ in $\Gal(K_\infty/K)$, since 
\[i_\fp( \sqrt{-q}) = ( \cdots, \underset{\fp}{\sqrt{-q}},\cdots ) \equiv  ( \cdots, \underset{\fq}{(\sqrt{-q})^{-1}} ,\cdots) \bmod \overline{K^\times \prod_{v\neq \fp}U_v}.\] 
Here the value is $1$ at the places under dots. Hence the decomposition group of $\fq$ in $K_\infty/K$ is isomorphic to
\[ 
(\langle \sqrt{-q}\rangle  \times \{\pm 1 \})/ \{\pm 1 \} =  ((1+2^k\CO_\fp)\times \{\pm 1\})/\{\pm 1 \}.
\]
 The latter group is clearly of index $2^{k-2}$ in $U_\fp/\{\pm 1\} $. Therefore, there are precisely $2^{k-2}$ primes of $K_\infty$ lying above $\fq$, completing the proof of this lemma. \end{proof}

\begin{prop}\label{prop: K_1}
The prime $\fp$ (resp. $\fp^*$) is inert in $K^*_\infty$ (resp. $K_\infty$) if $q\equiv 7\bmod 16$, and  the prime $\fp$ (resp. $\fp^*$) splits in $K^*_1$ (resp. $K^*_1$) if $q\equiv 15 \bmod 16$.
\end{prop}


\begin{proof}
Let $\CK_1=K_1K^*_1$, so that $\CK_1$ is the first layer of the unique $\BZ^2_2$-extension of $K$. Hence $\CK_1$ contains $K(\sqrt{2})$, the first layer of the cyclotomic $\BZ_2$-extension of $K$. Let $K'=\BQ(\sqrt{-2q})$. Let $\fp'$ and $\fq'$ be the unique prime of $K'$ above $2$ and $q$, respectively. We claim that the quartic extension $\CK_1/K'$ is unramified abelian. It is easy to see that $K'(\sqrt{2})/K'$ is unramified. Then since $\CK_1=K_1 K'(\sqrt{2})$ (resp. $\CK_1=K^*_1 K'(\sqrt{2})$), it follows that $\CK_1/K'(\sqrt{2})$ is unramified outside the primes above $\fp$ (resp. $\fp^*$), whence $\CK_1/K'(\sqrt{2})$ is unramified everywhere.  Since $\CK_1/\BQ$ is clearly Galois, the quartic extension $\CK_1/K'$ must be abelian. This proves the claim. Let $\Cl_{K'}$ be the class group of $K'$. Then, the Artin map induces an surjection $
\Cl_{K'}\twoheadrightarrow \Gal(\CK_1/K' )$.
Now, since $\fp'\fq'=(\sqrt{-2q})$ which is a principal ideal of $K'$, it follows that the product of the Frobenius elements of $\fp'$ and $\fq'$ is trivial. In particular, $\fp'$ and $\fq'$ have the same inertia degree in $\CK_1$. Now, since $q$ splits in $\BQ(\sqrt{2})$, it follows from Lemma~\ref{lem: proc} that $\fq'$ must have inertial degree $2$ in $\CK_1$ if $q\equiv 7\bmod 16$, and must split completely in $\CK_1$ if $q\equiv 15\bmod 16$. 
But, note that $2$ splits in $K$ and $\fp$ is ramified in $K_1$. It follows that, by counting the inertia degree of $2$ in $\CK_1$, $\fp$ must be inert in $K^*_1$, whence it is then necessarily inert in $K_\infty$ when $q\equiv 7\bmod 16$, and $\fp$ must split in $K^*_1$ when $q\equiv 15\bmod 16$.  An entirely similar argument proves the assertion about the behaviour of $\fp^*$ in $K_\infty$. \end{proof}

We remark that the above arguments also reprove an old result of Hasse \cite{Hasse}, which asserts that $4$ exactly divides $h_{K'}$ if $q\equiv 7\bmod 16$ and $8$ divides $h_{K'}$ if  $q\equiv 15\bmod 16$, where $h_{K'}$ is the class number of $K'=\BQ(\sqrt{-2q})$ as in the above proof. Let $\Cl_{K'}(2)$ be the $2$-primary subgroup of the class group of $K'$, which is a cyclic group by genus theory. Let $L(K')$ be the $2$-Hilbert class field of $K'$ so that the Artin map induces an isomorphism $\Cl_{K'}(2) \cong \Gal( L(K')/K').$ Firstly, the above proof shows that $4\mid h_{K'}$, since we just see that $\CK_1/K'$ is a quartic unramified abelian extension. Note that $\fq'$ has order $2$ in $\Cl_{K'}(2)$. Assume $q\equiv 15\bmod 16$. The above proof shows that $\fq'$ splits completely in $\CK_1$. Thus $L(K')$ must be bigger than $\CK_1$; otherwise, this contradicts to that $\fq'$ has order $2$ in the class group. This shows that $8$ divides $h_{K'}$. 
Now assume $q\equiv 7\bmod 16$. We need to show that $L(K')=\CK_1$. The above proof shows that the $\fq'$ has inertia degree $2$ in $\CK_1$. Since $L(K')/K'$ is a cyclic $2$-extension, writing $d=[L(K'):K']$, this would imply that $\fq'$ has inertia degree $2^{d-1}$. As $\fq'$ has order $2$ in the class group of $K'$, we must have $d=2$. Thus, we have $\CK_1=L(K')$ when $q\equiv 7\bmod 16$. This proves Hasse's result.

\begin{cor}\label{cor: K_1}
	Assume $\pi\CO_K = \fp^t$, where $t$ is the order of $\fp$ in the ideal class group of $K$. Then $\pi \equiv \pm 3 \bmod {\fp^*}^3$ if $q\equiv 7\bmod 16$, and $\pi \equiv \pm 1 \bmod {\fp^*}^3 $ if $q\equiv 15\bmod 16$. 
\end{cor}

\begin{proof}
Note that $t$ is odd by genus theory. Since $K_1/K$ is unramified outside $\fp$, we must have $K_1 \subset K(\sqrt{\pi},\sqrt{-1})$. But $K(\sqrt{-1})/K$ is ramified at both $\fp$ and $\fp^*$, and thus $K_1$ is equal to $K(\sqrt{\pi})$ or $K(\sqrt{-\pi})$. Without loss of generality, we may assume that $K_1=K(\sqrt{\pi})$. Note that $K(\sqrt{-\pi})$ must then be ramified at both $\fp$ and $\fp^*$. 
Let $\pi^*$ be the conjugate element of $\pi$, so that  $\pi\pi^* = 2^t$.
It follows from the oddness of $t$ that $K^*_1:=K(\sqrt{\pi^*})$ is unramified at $\fp$. Note that $\pi$ is a unit in $K_{\fp^*}=\BQ_2$. 
When $q\equiv 7\bmod 16$, it follows from Proposition~\ref{prop: K_1} that, $K_{\fp^*}(\sqrt{\pi})$ or $K_{\fp^*}(\sqrt{-\pi})$ is an unramified quadratic extension of $\BQ_2$, whence $\pi \equiv \pm 3\bmod {\fp^*}^3$. When $q\equiv 15\bmod 16$,  we have that $ \sqrt{\pi} $ or $\sqrt{-\pi}$ is in $K_\fp^{*}$ and hence $\pi \equiv \pm 1\bmod {\fp^*}^3$.
\end{proof}

\begin{lem}\label{lem: ram F/K}
	There exists exactly one ramified prime $w$ of $F$ above $2$, which lies above $\fp$ by our notational convention.  Then $\fp^*$ is inert in $F$ when $q\equiv 7\bmod 16$ , and $\fp^*$ splits in $F$ when $q\equiv 15\bmod 16$.
\end{lem}
\begin{proof}
	Note that $-q$ is a square in $\BQ_2$. We have that $\sqrt{-q} \equiv \pm 3 \bmod 8$ if $q\equiv 7\bmod 16$ and that $\sqrt{-q}\equiv \pm 1\bmod 8$ if $q\equiv 15\bmod 16$.  By adjusting the signs, the polynomial \eqref{pol} has the following factorization in $\BQ_2[x]$:
	\[ f(x) =   \begin{cases}
(x^2+\sqrt{-q})(x^2-\sqrt{-q}) & \text{ if } q\equiv 7\bmod 16;\\
(x^2+\sqrt{-q})(x-\sqrt[4]{-q})(x+\sqrt[4]{-q}) & \text{ if } q\equiv 15\bmod 16.
	\end{cases}\] 
Thus, the lemma follows from the fact that $\BQ_2(\sqrt{3})$ and $\BQ_2(\sqrt{-1})$ are ramified over $\BQ_2$ and that $\BQ_2(\sqrt{-3})/\BQ_2$ is unramified over $\BQ_2$. \end{proof}

In what follows,  we write $w$ for the prime of $F$ above $\fp$, and $w^*$ for any of the primes of $F$ above $\fp^*$. The following proposition establishes a lower bound for $ \ord_{w^*}(\log_{w^*}(\eta))$, which is valid for all primes  $q\equiv 7\bmod  8$.  It is only this weaker form of Theorem \ref{thm: log} which is needed for the proof of Theorem \ref{mir} when $q \equiv 7 \mod 16.$ Again there is a completely parallel result for the field $F'$.

\begin{prop}\label{prop: log lower bound}
 Assume $q\equiv 7\bmod 8$, and let $w^*$ be any prime of $F$ above $\fp^*$. Writing $\eta$ for a fundamental unit of $F$, we have
$ \ord_{w^*}(\log_{w^*}(\eta)) \geq 3$.
\end{prop}

\begin{proof}
	Let $\CO_F$ be the ring of integers of $F$, and let $h_F$ be its class number. Then $h_F$ is odd (see Proposition~ \ref{prop: clno1}, or Theorem 3.8 of \cite{CL}). We define $\gamma$ to be the cube of a generator of the ideal $w^{h_F}$ of $\CO_F$, where, as always,  $w$ is the ramified prime of $F$ above $\fp$, and put $\beta=N_{F/K}(\gamma)$. Since $h_F$ is odd,  $\beta$ is a generator of an odd power of $\fp$. Note also that $\beta\CO_F=\gamma^2\CO_F$ because $w$ is the unique prime of $F$ above $\fp$.. Hence $\gamma^2/\beta$ is a unit of $F$, whence it equals $\pm \eta^k$ for some integer $k$. The integer $k$ must be odd, otherwise $\sqrt{\pm \beta}\in F$, which is impossible since it would imply that $F/K$ is ramified only at primes dividing 2, contradicting the fact that the unique prime of $K$ above $q$ is ramified in $F$. Put $\eta' = \gamma^2/\beta$. Since $k$ is odd, we have
	$\ord_{w^*} \log(\eta') = \ord_{w^*} \log(\eta)$, so it suffices to prove that $\ord_{w^*} \log(\eta') \geq 3$.
	By Corollary~\ref{cor: K_1}, we have
	\begin{equation} \label{eq: pi}
	\beta \equiv \begin{cases}
	\pm 3 \bmod {\fp^*}^3 &\text{ if } q\equiv 7\bmod 16,\\
	 \pm 1 \bmod {\fp^*}^3 &\text{ if } q\equiv 15\bmod 16.\\
	\end{cases}\end{equation}
Assume first that  $q\equiv 7\bmod 16$. It is proven in \cite[Theorem 3.1]{CL}  that $M(F_\infty)=F_\infty$ whence the maximal abelian  $2$-extension of $F$ having exponent $2$, which is unramified outside $\fp$, is $F_1 = FK_1$. Moreover, since $\beta$ is an odd power of $\pm \pi$, it is shown in the proof Proposition \ref{prop: K_1} that $F_1=F(\sqrt{\beta})$ or $F_1 = F(\sqrt{-\beta})$. We claim that the two fields $F(\sqrt{\gamma})$ and $F(\sqrt{-\gamma})$  must both be ramified at $w^*$. Indeed, if  $F(\sqrt{\gamma})$ is unramified at $w^*$, it would have to be equal to $F_1$, whence it would follow that the ideal $\gamma \CO_F$ is a square, which it is not because $\gamma\CO_F = w^{3h_F}$ and $h_F$ is odd.  A similar argument shows that $F(\sqrt{-\gamma})$ is ramified at $w^*$. From now on, we will work in the completion $F_{w^*}$ of $F$ at $w^*$,  and for simplicity, we write $a$ for $\iota_{w*}(a)$, where $\iota_{w^*}$ is the canonical embedding of $F$ into $F_{w^*}$.
	The ring of integers  $\CO_{w^*}$ of $F_{w^*}$ is $\BZ_2[\zeta]$ where $\zeta$ is a primitive cube root of unity, and $2$ is a local parameter for this ring. The residue field $\CO_{w^*}/2\CO_{w^*}$ is isomorphic to the finite field with $4$ elements, and $\{0, 1, \zeta,\zeta^2 \} $ is a system of representative elements for this residue field.  Thus any element of $1+2\CO_{w^*}$ can be written uniquely as a series $1+a_1 2+a_2 2^2+a_3 2^3 \cdots $, with $a_i \in \{0,1, \zeta,\zeta^2\}$.
	Now  $\gamma$ is a unit in $\CO_{w^*}$, and in fact lies in $1+2\CO_{w^*}$ since we defined  $\gamma$ to be a cube in $F$. Since $F(\sqrt{\gamma})/F$ and $F(\sqrt{-\gamma)}/F$ are both ramified at $w^*$, it follows that $\gamma\not\equiv \pm 1 \bmod 4$. Hence we must have
	\[ \gamma\equiv 1+2\zeta  \quad \text{ or }  \quad 1+2\zeta^2 \bmod 4\CO_{w^*},\]
	and so it follows that 
	\[ \gamma^2 \equiv 1+4\zeta^2+4\zeta \equiv -3\bmod 8. \]
	Together with \eqref{eq: pi}, this shows that 
	\[ \eta'  = \frac{\gamma^2}{\beta} \equiv \pm 1 \bmod 8\CO_{w^*}. \]
	Thus $ {\eta'}^2-1$ is divisible by $16$ in $\CO_{w^*}$. Therefore we have
	$ \ord_{w^*}(\log_{w^*}({\eta'}^2) ) \geq  4 $  whence $\ord_{w^*}(\log_{w^*}(\eta') ) \geq 4-\ord_{w^*}(2)=3. $
	This proves the case when $q\equiv 7\bmod 16$.
	
	Suppose next that $q\equiv 15\bmod 16$ so that $F_{w^*}=\BQ_2$ by Lemma \ref{lem: ram F/K}. It is proven in \cite{L1} that the maximal abelian  $2$-extension $T$ of $F$ having exponent $2$, which is unramified outside $\fp$, is strictly larger than $F_1$. Noting that $F(\sqrt{\eta'})=F(\sqrt{\beta})$ and $h_F$ is odd, we must have $T\subset F(\sqrt{-1},\sqrt{\gamma},\sqrt{\beta})$. But $F(\sqrt{-1})$ is clearly not contained in $T$ as $F_{w^*}(\sqrt{-1})=\BQ_2(\sqrt{-1})$ is ramified over $F_{w^*}=\BQ_2$. This implies that $F(\sqrt{\gamma})$ or $F(\sqrt{-\gamma})$ is unramified at the primes above $\fp^*$. It follows that $\gamma^2 \equiv 1 \bmod {w^*}^3.$
	Together with \eqref{eq: pi}, this shows that  $\eta'  = \frac{\gamma^2}{\beta} \equiv \pm 1 \bmod 8\CO_{w^*}.$
	Thus $\ord_{w^*}(\log_{w^*}(\eta'))\geq 3$, completing the proof of Proposition~\ref{prop: log lower bound}.
	\end{proof}

By a classical formula \cite{CW}, whose proof we recall in the Appendix, Proposition~\ref{prop: log lower bound} and its analogue for the field $F'$, are equivalent to  
\begin{equation}\label{eq: M(F')}
[M^*(F):F^*_\infty]\geq 4 \text{ and } [M(F'):F'_\infty]\geq 4.
\end{equation}

\noindent Note that $J = F(\sqrt{-1})/F$ is unramified outside $w^*$ for any $q\equiv 7\bmod 8$, since $F_w(\sqrt{-1})/F_w$ is unramified.  But we claim that $J$ is not equal to $FK^*_1$. Indeed, if $J = FK^*_1$,  it would follow that $-\pi^*$ is a square in $F$, since $FK^*_1=F(\sqrt{\pi^*})$ as in the proof of Proposition~\ref{prop: K_1};  here $\pi^*\in \CO_K$ is a generator of some odd power of $\fp^*$. Then $F$ would coincide with the field $K(\sqrt{-\pi^*})$,  and so the prime $\fq$ of $K$ would be unramified in $F$, which is plainly a contradiction.
We remark that $F^*_{\infty}(\sqrt{-1})$ (resp. $F'_{\infty}(\sqrt{-1}$)) is in fact the unique quadratic extension of $F^*_{\infty}$ (resp. $F_\infty$) lying inside $M^*(F)$ (resp. $M(F')$) when $q \equiv 7 \mod 16$ since $X^*(F)=\Gal(M^*(F)/F_\infty^*)$ and $X(F')=\Gal(M(F')/F'_\infty)$ are both cyclic groups of order $4$ by Theorem \ref{mir2}. However, it does not seem obvious to write down explicitly the  corresponding cyclic quartic extension, whose existence is proven by Theorem \ref{mir2}. Nevertheless, we now show that we obtain the striking Theorem~\ref{mir} from \eqref{eq: M(F')}. 

\begin{proof}[Proof of Theorem~\ref{mir}]
Note that $F'_\infty(\sqrt{-1})=J_\infty$ is contained in $M(F')$. It follows from \eqref{eq: M(F')} and the fact that $X(F')=\Gal(M(F')/F'_\infty)$ is of order at least $4$ (in fact it is cyclic of order exactly $4$ by Theorem \ref{mir2}) that
	\begin{equation}\label{eq: m1}
	M(J) \neq J_\infty.
	\end{equation}
Now, writing $\Gamma=\Gal(J_\infty/J)$ and $(X(J_\infty))_\Gamma$ for the $\Gamma$-coinvariants of $X(J_\infty)$, it is easily seen that we have
	\begin{equation}\label{eq: m2}
	(X(J_\infty))_\Gamma = \Gal(M(J)/J_\infty).
	\end{equation}
	We conclude from \eqref{eq: m1} and \eqref{eq: m2} that $X(J_\infty) \neq 0$. Thus, to complete the proof of the theorem, it suffices to show that $X(J_\infty)$ is a free $\BZ_2$-module of rank at most $1$. We shall make again essential use of the important fact, which is proven in \cite{CL}, Theorem 3.1, that $M(F_\infty) = F_\infty$ for all primes $q \equiv 7 \mod 16$. Let $\Delta = \Gal(J_\infty/F_\infty)$, so that $\Delta= \langle \delta\rangle$, where $\delta$ is of order $2$. Then $\Delta$ acts on $X(J_\infty)$ as usual by lifting inner automorphisms.  It follows from Proposition~\ref{prop: K_1} that the unique prime $w^*$ of $F$ above $\fp^*$ splits into two primes in $F_1$, and these two primes are then inert in $F_\infty/F_1$. Since $w^*$ is ramified in $J$, it follows that there are exactly two primes of $F_\infty$ which ramify in $J_\infty$, and both of these primes have ramification index 2 in the extension $M(J_\infty)/F_\infty$ since they are unramified in $M(J_\infty)/J_\infty$. 
Let $R$ be the maximal abelian extension of $F_\infty$ which is contained in $M(J_\infty)$, so that
$$
\Gal(R/J_\infty) = X(J_\infty)_\Delta,
$$
where $X(J_\infty)_\Delta$ denotes the $\Delta$-coinvariants of $X(J_\infty)$. Now, since $M(F_\infty) = F_\infty$ by Theorem 3.1 of \cite{CL}, it follows that $\Gal(R/F_\infty)$ must be generated by the inertial subgroups of the two primes of $F_\infty$
lying above $w^*$, both of which are of order 2. Hence $\Gal(R/F_\infty)$ is of order at most 4, and so $X(J_\infty)_\Delta$ is of order at most 2. Hence, by Nakayama's lemma,
$X(J_\infty)$ is generated by one element, say $x$,  over the group ring $\BZ_2[\Delta]$. Now we have the exact sequence
$$
0 \to X(J_\infty)^\Delta \to X(J_\infty) \to X(J_\infty) \to X(J_\infty)_\Delta \to 0,
$$
where the middle map is multiplication by $\delta - 1$. Since all of these groups are finitely generated $\BZ_2$-modules and $X(J_\infty)_\Delta$ is finite by the Nakayama lemma argument we have just given, it follows that $X(J_\infty)^\Delta$ must be finite. But Greenberg (see his remarks at the end of \S 4 of \cite{RG3}), has shown that $X(J_\infty)$ has no non-zero finite submodules which is stable under the action of $\Gal(J_\infty/J)$, and so necessarily $X(J_\infty)^\Delta = 0.$ In particular, it follows that $(1+\delta)X(J_\infty) =0$. Hence $X(J_\infty) = [\BZ_2[\Delta]/(1+\delta)]x$,
and so we see that $X(J_\infty)$ has at most rank 1 as a $\BZ_2$-module. Again invoking Greenberg's theorem, it follows that $X(J_\infty)$ must be a free $\BZ_2$-module
of rank at most 1, and so finally it must be free of rank 1 because $X(J_\infty) \neq 0$. The assertion on $X^*(J^*_\infty)$ now follows from Lemma~\ref{lem: symmetry}. This completes the proof.\end{proof}
	 
We now turn to the proof of  Theorem~\ref{thm: log}. Thanks to Proposition~\ref{prop: log lower bound}, we only need to show that $\ord_{w^*}(\log_{w^*}(\eta)) \leq 3$, always assuming that $q\equiv 7\bmod 16$.  Unlike the arguments used in the proof of Proposition \ref{prop: log lower bound}, we shall prove this upper bound by working with Galois groups. Our first lemma determines the group structure of $X^*(F)$ using class field theory, where we recall that, by definition \eqref{a1}, $X^*(F)=\Gal(M^*(F)/F^*_\infty)$.

\begin{lem}\label{lem: cyclic}
	Assume $q\equiv 7\bmod 16$. Then $X^*(F)$ and $X(F')$ are both finite cyclic groups.
\end{lem}

\begin{proof} 
	We only need to prove the assertion for $X^*(F)$, since $X^*(F) \cong X(F')$ by Lemma~\ref{lem: symmetry}. Let $ U_1=1+2\CO_{w^*}$ be the group of principal units of $F_{w^*}$. Since the class number of $F$ is odd by Proposition~\ref{prop: clno1}, it follows from global class field theory that
	\[ U_1/\overline{\langle {\eta, -1 }\rangle}\cong \Gal(M^*(F)/F). \]
	Here the bar denotes the closure in $U_1$ under the $2$-adic topology. Note that the group of $2$-power roots of units inside $F_{w^{*}}=\BQ_2(\sqrt{-3})$ is $\{\pm 1 \}$. It follows that the left hand side is isomorphic to 
	\[   \log_{w^*} (U_1)/\BZ_2 \log_{w^*}(\eta). \]
	Since $\log_{w^*}(U_1)\subset \CO_{w^*}\cong \BZ^2_2$, we must have $ \log_{w^*} (U_1)\cong \BZ^2_2$ as $\BZ_2$-modules. Thus the $\BZ_2$-module $\Gal(M^*(F)/F)$ is generated by two elements and must be isomorphic to $\BZ_2\times \BZ/{2^r}\BZ$ for some integer $r$.
	Therefore $ X^*(F)$, being the $\BZ_2$-torsion submodule of $\Gal(M^*(F)/F)$ is a finite cyclic $2$-group, and the assertion for $X(F')$ follows from Lemma~\ref{lem: symmetry}. \end{proof}
\medskip

We recall that $h$ denotes the class number of $K$, and, as always, $\eta$ is a fundamental unit of the field $F$.

\begin{lem}\label{lem: elements out of M*}
	Assume $q\equiv 7\bmod 16$. Write $i = \sqrt{-1}$, and let $\pi^*$ be a generator of the ideal ${\fp^*}^{h}$ of $K$.  Then none of $\sqrt{i}, \sqrt{i\eta},\sqrt{i \eta \pi^*} , \sqrt{\eta \pi^*}$ lie in the field $M^*(F)$.
\end{lem}

\begin{proof}
	As we remarked under \eqref{eq: M(F')}, $J$ is contained in $M^*(F)$.	
	Thus, to prove the first two elements are not in $M^*(F)$, it suffices to show that the extensions $J(\sqrt{i})$ and $ J(\sqrt{i\eta})$ of $J$ both are ramified at the unique prime $v$ of $J$ lying above $w$,  where $w$ is the ramified prime of $F$ above $\fp$. Note that $J_v=\BQ_2(\sqrt{-3},i)$, so that the extension $J_v(\sqrt{i})/J_v$ is obviously ramified, and so $\sqrt{i} \notin M^*(F)$. Next, we recall that it was shown in the proof of Proposition \ref{prop: log lower bound} that $\eta \equiv \pi  \bmod (J^\times)^2$, where $\pi$ is a generator of $\fp^t$ and $t$ is the order of $\fp$ in the class group of $K$. Moreover, $\pi \equiv \pm 3 \bmod {\fp^*}^3$ by Corollary~\ref{cor: K_1}.  Thus $\sqrt{\pi} \in J_v$, and so $J_v(\sqrt{i\eta})=J_v(\sqrt{i})$ is also ramified over $J_v$. This proves $\sqrt{i\eta}\notin M^*(F)$. If $\sqrt{i\eta \pi^*}\in M^*(F)$, it would imply that $\sqrt{i\eta}\in M^*(F)$, since $JK^*_1=J(\sqrt{\pi^*})\subset M^*(F)$. 
	Thus $\sqrt{i\eta \pi^*} \notin M^*(F)$. Finally, noting that $J(\sqrt{\eta \pi^*}) = J(\sqrt{\pi \pi^*}) =J(\sqrt{2})$ and that $J_{v}(\sqrt{2})/J_{v}$ is ramified, we conclude that $\sqrt{\eta \pi^*}\notin M^*(F)$. This completes the proof of Lemma~\ref{lem: elements out of M*}.\end{proof}

\medskip

\begin{proof}[Proof of Theorem~\ref{thm: log}]
	We recall that $\fp^*$ is inert in   
	$F$ by Lemma \ref{lem: ram F/K}, and that $J \subset M^*(F)$. Thus, applying Theorem \ref{cw} to the extension $F/K$ at the prime $\fp^*$, we conclude from Proposition~\ref{prop: log lower bound} and Lemma~\ref{lem: cyclic} that
	\begin{equation} \label{eq: M*/F}
	\Gal(M^*(F)/F)\cong \BZ_2\times \BZ/{2^r}\BZ \text{ for some integer } r\geq 2. 
	\end{equation}
	Our goal is to show that $r$ is equal to $2$. Suppose, on the contrary, that $r\geq 3$. Hence, by Galois theory, there are intermediate fields $T$ and $E$ of $M^*(F)/F$ such that 
	\[ T\supsetneq E \supsetneq J \supsetneq F  \text{ and } \Gal(T/F)\cong \BZ/8\BZ. \]
	Since $\sqrt{-1}\in J$, according to Kummer theory, there exists some $\xi \in J$ such that 
	\[ T=J(\sqrt[4]{\xi}) \text{ whence }E=J(\sqrt{\xi}).\]  
	\medskip
	Now we claim that,  since $T/F$ is a cyclic extension of degree $8$, we  must have 
	\begin{equation}\label{eq: fieldsT,E}
	T = F(\sqrt[4]{\xi}) \text{ and } E=F(\sqrt{\xi}).
	\end{equation}
	Indeed, by Galois theory,  the intermediate field $F(\sqrt[4]{\xi})$ must be equal to one of the fields $F, J, J(\sqrt{\xi}), J(\sqrt[4]{\xi})$, and it is easily seen that it cannot be equal 
	to any of the fields $F, J, J(\sqrt{\xi})$.
	We recall that $w^*$ denotes the unique prime of $F$ lying above $\fp^*$, and we let $v^*$ be the unique prime of $J$ above $w^*$. Since $w^*$ is ramified in $J$,  we have $w^*\CO_J= {v^*}^2$ where $\CO_J$ denotes the ring of integer of $J$. Now ${v^*}^{h_J}=\varpi \CO_J$ for some $\varpi\in \CO_J$, where $h_J$ denotes the class number of $J$. Since $h_J$ is odd by Proposition~\ref{prop: clno1}, and $T/J$ is unramified outside $v^*$, we conclude that 
	\begin{equation}\label{eq: form of xi}
	\xi \equiv u \varpi^k \bmod (J^\times)^4 \text{ for some  } u \in \CO^\times_J \text{ and  } k =0, 1,2, \text{ or }3.
	\end{equation}
	Write $\tilde{\delta}$ for a generator of $\Gal(T/F)$, and let $\delta$ be the image of $\tilde{\delta}$ in $\Gal(J/F)$. For simplicity, write ${x'}$ for $\delta(x)$ if $x\in J$. 
	Now since $T/F$ is cyclic, the field $T=J(\sqrt[4]{\xi})$ is also equal to $J(\sqrt[4]{{\xi'}})$, and the field $E=J(\sqrt{\xi})$ is also equal to $J(\sqrt{{\xi'}})$. Thus, we have
	\begin{equation}\label{eq: xi bar xi} 
	{\xi'}\equiv \xi \bmod (J^\times)^4 \text{ and  } \xi{\xi'} \in (J^\times)^2. 
	\end{equation}  
	We claim that the integer $k$ in \eqref{eq: form of xi} must be $0$ or $2$. Suppose, on the contrary, that $k$ is odd. Then  $\sqrt{\xi{\xi}'} \notin F$, since the ideal $\xi{\xi'} \CO_F$ is an odd power of $w^*$ by the oddness of $h_J$ and $k$. Thus $J$ must coincide with $F(\sqrt{\xi{\xi'}})$, since $\sqrt{\xi{\xi'}}\in T$ and $J$ is the unique quadratic extension of $F$ inside $T$. It would then follow that $\xi{\xi'} \equiv -1 \bmod (F^\times)^2$, which is impossible because $\xi{\xi'}\CO_F = {w^*}^k$, with $k$ is odd. Thus $k$ must be even.
	
	\medskip
	
	Next we note that we must have
	\[ \xi{\xi'} =  \pm \eta^m  \text{ or } \pm \eta^m {\pi^*}^2 \bmod (F^\times)^4 \text{ for some } m \in \{0,1,2,3\}.     \]
	Now the integer $m$ must be equal to $0$ or $2$, since, if this is not the case,  $\sqrt{\pm \eta} \in J$,  and so we would have $\mp \eta \in (F^\times)^2$ because $J=F(\sqrt{\pm \eta})=F(\sqrt{-1})$. But this is clearly impossible since $\eta$ is a fundamental unit of $F$, whence $m$ is $0$ or $2$.
	Next we claim that $\xi{\xi'}$ is not a square in $F$. Suppose the contrary, and  write ${\xi'}=a^2/\xi$ for some $a\in F$. Then $\tilde{\delta}(\sqrt{\xi})= \pm \sqrt{{\xi'}}= \pm a/\sqrt{\xi}$, whence $\tilde{\delta}^2 (\sqrt{\xi})= \sqrt{\xi}$.
	Since $E=F(\sqrt{\xi})$ by \eqref{eq: fieldsT,E}, this implies that $\tilde{\delta}|_E$ has order $2$ in $\Gal(E/F)$. But this contradicts the fact that $\tilde{\delta}|_E$ is a generator of the cyclic group $\Gal(E/F)$ of order $4$. Therefore, we must have  
	\[\xi{\xi'} \equiv -1, -\eta^2, -{\pi^*}^2,  \text{ or } -\eta^2 {\pi^*}^2 \bmod (F^\times)^4.\]
	It follows from \eqref{eq: xi bar xi} that necessarily
	\[  \xi^2 \equiv -1, -\eta^2, - {\pi^*}^2,  \text{ or } -\eta^2 {\pi^*}^2 \bmod (J^\times)^4.\]
	Hence, we have 
	\[  \xi  \equiv  \sqrt{-1}, \quad   \eta \sqrt{-1}, \quad  \pi^* \sqrt{- 1},   \text{ or }    \eta \pi^* \sqrt{-1}  \bmod (J^\times)^2.\]
	But from Lemma~\ref{lem: elements out of M*}, it would then follow that $\sqrt{\xi} \notin M^*(F)$. This contradiction shows that the integer $r$ in \eqref{eq: M*/F} is indeed equal to $2$. The proof is complete. \end{proof}

In terms of Galois groups, Theorem~\ref{thm: log} is equivalent to the following statement by Theorem~\ref{cw}.
\begin{thm}\label{mir2}
	Assume $q\equiv 7\bmod 16$. Then $X(F)= X^*(F')=0$ and $X^*(F)= X(F')=\BZ/4\BZ$.
\end{thm}

\medskip

\begin{lem}\label{lem: 15mod}
	Assume $q\equiv 15\bmod 16$. Then $X^*(F)$ and $X(F')$ are both isomorphic to $\BZ/2\BZ \times \BZ/2^r\BZ$ for some $r\geq 1$.
\end{lem}

\begin{proof} 
	We only need to prove the assertion for $X^*(F)$ by Lemma~\ref{lem: symmetry}. According to Lemma~\ref{lem: ram F/K}, $\fp^*$ splits in $F$, and we
	write $w^*_1$ and $w^*_2$ for the two primes of $F$ above $\fp^*$. Then, for $i=1, 2$, the completion $F_{w^*_i}=\BQ_2$, we let $\CO_{w_i}$ denote the ring of integer of $F_{w^*_i}$, and we write $\iota_i$ for the canonical embedding from $F$ to $F_{w^*_i}$. Let $U_1 = \CO^\times_{w^*_1} \times \CO^\times_{w^*_2}$. 	 Since the class number of $F$ is odd by Proposition~\ref{prop: clno1}, it follows from global class field theory that
	\[ U_1/\langle {\overline{(\iota_1(\eta),\iota_2(\eta)), (-1,-1) }} \rangle\cong \Gal(M^*(F)/F), \]
	where, as always $\eta$ denotes a fundamental unit of $F$. The logarithm map $(\log_{w^*_1}, \log_{w^*_2})$ then gives rise to an isomorphism of the group on the left to the 
	additive group
	\[   (\{\pm1\} \times \{\pm 1\})/\langle (-1,-1) \rangle \times (4\CO_{w^*_1} \times 4\CO_{w^*_2}) / \overline{ \langle ( \log_{w^*_1}(\eta), \log_{w^*_2}(\eta)) \rangle}.  \]
	Now Proposition~\ref{prop: log lower bound} shows that  $\log_{w^*_i}(\eta) $ is divisible by $8$. Hence there is a an integer $r \geq 1$ such that the above group is isomorphic to $\BZ_2 \times \BZ/2\BZ \times  \BZ/{2^r\BZ}.$
	Thus $X^*(F)$, being the torsion submodule of $\Gal(M^*(F)/F)$, is isomorphic to $\BZ/2\BZ \times \BZ/{2^r\BZ}$ for some $r\geq 1$. By Lemma~\ref{lem: symmetry}, we have the same result for $X(F')$. \end{proof}

\begin{cor} Assume $q \equiv 15 \mod 16$. Then $X^*(F^*_\infty)$ and $X(F'_\infty)$ are both free finitely generated $\BZ_2$-modules of rank at least 2. 
\end{cor}

We have carried out numerical computations for all primes $q\equiv 15\bmod 16$ with $q< 10000$, and these computations show that, in this range, we have $\ord_{w^*_i}\log_{w^*_i}(\eta) \geq 4$ for ($i=1,2$), proving that the integer $r$ in Lemma~\ref{lem: 15mod} is at least $2$, and is sometimes equal to $2$. The numerical data also seems to indicate that the integer $r$ in  Lemma~\ref{lem: 15mod} is unbounded when $q$ varies in the family of primes $\equiv 15 \bmod 16$.  Always assuming that $q\equiv 15\bmod 16$, it is shown in \cite{L1} that $\ord_{w}(\log_{w}(\eta)) \geq 4$, where $w$ now denotes the unique ramified prime of $F$ above $2$, whence $X(F)\neq 0$ by Theorem~\ref{cw}. In fact, one can refine the proof used given in \cite{L1} to show that $\ord_{w}(\log_{w}(\eta)) \geq 6$ for all primes $\equiv 15 \bmod 16$, whence $X(F)$ is always a cyclic group of order at least $4$. Zhibin Liang had computed the order of  $X(F)$ for all primes $q\equiv 15\bmod 16$ and $q< 2500$. His calculations show that this cyclic group sometimes has order exactly $4$,  but  they also seem  to indicate that there is no upper bound for its order when $q$ varies.

\section{Proof of Theorems~\ref{thm: Y}}

In addition to the proof of Theorem~\ref{thm: Y}, our goal of this section also includes computing $Y(R), Y(R_\infty)$ and $Y(\CR_\infty)$ for $R= D, F, F' $ and $J$. As in Lemma~\ref{lem: symmetry}, we have $Y(F) \cong Y(F'), Y(F_\infty)\cong Y(F'^*_\infty), Y(F'_\infty)\cong Y(F^*_\infty)$ and $Y(R_\infty)\cong Y(R^*_\infty)$ for $R=K, D,$ and $J$.
We first recall some classical results which will be needed. Let $M$ be a number field, and let $S$ be any finite set of prime ideals of $M$.
We write $\Cl_{M,S}$ for the quotient of the ideal class group of $M$ by the subgroup generated by the ideal classes of primes in $S$. Similarly, if $M'$ is any finite extension of 
$M$, $\Cl_{M',S}$ will denote the quotient of the ideal class group of $M'$ by the subgroup generated by the classes of primes of $M'$ lying above $S$. Also, we write $\CO_{M,S}^\times$ for the group of $S$-units of $M$. We begin with Chevalley's classical formula, which first occurs in his thesis \cite[p.~406]{Cheva} (for a modern proof, see \cite{LY}).

\begin{prop}\label{Ch}   Let $M'/M$ be a finite cyclic extension of number fields with Galois group $G$, and let $S$ be an arbitrary finite set of prime ideals of $M$. Then
\begin{equation}\label{eq: amb}
\#((\Cl_{M',S})^G) =  \#(\Cl_{M,S})\frac{\prod_{v\in S}e_v f_v \prod_{v\notin S}e_v}{[M':M][\CO^\times_{M,S}:\CO^\times_{M,S}\cap N_{M'/M}({M'}^\times)]}, 
\end{equation}
where $e_v$ (resp. $f_v$) denotes the ramification index (resp. inertial degree) of a prime $v$ in the extension $M'/M$, and $N_{M'/M}$ is the norm map.
\end{prop}
\noindent  In the present paper, we shall always use this formula in the case when $S$ is the empty set.

\medskip

 Let $p$ be any prime. Let $T_n/T_0$ be a cyclic extension of number fields of degree $p^n$. For $0\leq i\leq n$, let $T_i$ be the unique intermediate field such that $[T_i:T_0]=p^i$. We assume that every  prime of $T_0$ ,which ramifies in $T_n$, is in fact totally ramified, and that there is at least one such ramified prime. Let $S$ be a finite set of primes of $T_0$, such that the decomposition subgroup in $T_n/T_0$ of every prime in $S$ is either $0$ or the whole of $\Gal(T_n/T_0)$. Thus, every prime in $S$ is either totally ramified, or splits completely, or is inert  in $T_n$. We write $\CG = \Gal(T_n/T_0)$.
 \begin{prop}\label{prop: stable}
 With the above assumptions,  let $A_i$ be the $p$-primary part of the $S$-class group of $T_i$ for $0\leq i\leq n$. If $m \geq 0$ is an integer such that $\#(A_0/p^mA_0)= \#(A_1/p^mA_1)$, then
 $$  
	 A_0/p^mA_0 \cong  A_r/p^mA_r  \quad \text{ for any } 0\leq  r \leq n. 
$$ 
In particular, if $\#(A_0)=\#(A_1)$, then $A_0\cong A_r$ for $0\leq r\leq n$.
\end{prop}

In particular, Proposition \ref{prop: stable} immediately yields the following result about $\BZ_p$-extensions.

\begin{prop}\label{prop: stable in Z_p}
	Let $T$ be a number field, and let $T_\infty/T$ be a $\BZ_p$-extension. Let $n_0\geq 0$ be such that every prime ramified in $T_\infty/T$ is totally ramified in $T_\infty/T_{n_0}$. Let $S$ be a finite set of primes of $T_{n_0}$ such that the decomposition group of each prime in $S$ in the extension $T_\infty/T_{n_0}$  is either  $\Gal(T_\infty/T_{n_0})$ or $0$. Write $A_k$ for the $p$-primary subgroup of the $S$-class group of $T_k$ for $k\geq n_0$. If $m \geq 0$ is an integer such that $\#(A_{n_0}/p^m A_{n_0})  =  \#(A_{n_0+1}/p^mA_{n_0+1})$, then 
	\[ A_{n_0}/p^mA_{n_0} \cong  A_{n_0+r}/p^mA_{n_0+r} \text{ for all r} \geq 0.    \]
	In particular, $\#(A_{n_0})=\#(A_{n_0+1})$ implies that $A_{n_0}\cong A_{n_0+r}$ for all $r\geq 0$.
\end{prop}

 Fukuda \cite{Fuk94} proves this result for ideal class groups and for $m=0,1$. 
Now $A_n$ is a $\BZ_p[\CG]$-module, where, as above $\CG = \Gal(K_n/K_0)$. We give a rough bound on the number of generators of $A_n$ as a $\BZ_p[\CG]$-module in terms of $A_0$. This in turn gives a rough bound on the number of generators of $A_n$ as a $\BZ_p$-module. For a $\BZ_p$-module $A$, we write $\mathrm{rk}_p(A)$ for the $\BF_p$-dimension of $A/{pA}$.  Let $N_\CG= \sum_{\gamma \in \CG}\gamma \in \BZ_p[\CG]$.
\begin{prop}\label{prop: bound clgp}
	We use the same notation as in Proposition~\ref{prop: stable}, and write $g_n$ for the minimal number of generators of $A_n$ as a $\BZ_p[\CG]$-module, and  put $g_0= \mathrm{rk}_p (A_0)$. 
	Write $t$ for the number primes of $T_0$ which are either ramified in the extension $T_n/T_0$ or lie in $S$ and are inert in the extension $T_n/T_0$.Then $ g_n \leq  g_0  + t-1$ and $\mathrm{rk}_p(A_n) \leq p^n (g_0 + t -1)$; if further $N_\CG A_n=0$, then $\mathrm{rk}_p(A_n)\leq (p^n-1)(g_0 +t -1)$.
\end{prop}

\noindent We remark that if $p\nmid \#(A_0)$, then automatically $N_\CG A_n=0$. We briefly explain Proposition \ref{prop: bound clgp} is a rough bound. 
Let us take $p=2$ and $n=1$, and assume that $T_0$ has odd class number. By Proposition~\ref{prop: bound clgp}, $\mathrm{rk}_2(A_1)\leq t -1$. But now $A^\CG$, and $\CG =\Gal(T_1/T_0)$, coincides with the subgroup of $A_1$ killed by $2$, and so  $\mathrm{rk}_2{A_1}=\mathrm{rk}_2{A^\CG}$.
It follows from Chevalley's formula \eqref{eq: amb} that 
\[  \mathrm{rk}_2 A_1 = t-1 - \mathrm{rk}_2 \left( \CO^\times_{T_0,S}/\CO^\times_{T_0,S} \cap N_{T_1/T_0}(T^\times_1)\right).\]

\bigskip

The proof of these results is a slight modification of classical arguments in Iwasawa theory, going back to Iwasawa's paper  \cite[\S 7.4]{Iwa}. However, for completeness we now give a self-contained proof of them.

\begin{proof}[Proof of Proposition~\ref{prop: stable}]
	Let $L_i$ be the maximal unramified abelian $p$-extension of $T_i$, in which the primes of $T_i$ above $S$ splits completely. By class field theory, the Artin map induces an isomorphism
	\[  A_i \cong \Gal(L_i/ T_i ).  \]
	Let $Y_i= \Gal(L_i/T_i )$, and for simplicity, we put  $L=L_n$ and $Y=Y_n$. The group $\CG= \Gal(T_n/T_0)$ acts on $Y$ by lifting inner automorphisms, so that $Y$ becomes a $\BZ_p[\CG]$-module, and then $Y\cong A_n$ as $\BZ_p[\CG]$-modules.
	Let \begin{equation}\label{eq: Z}
	 Z= \Gal(L/ T_nL_0)
	\end{equation} 
	which is a $\CG$-submodule of $Y$. Now since there is at least one totally ramified prime in the extension $T_n/T_0$, we have $L_0\cap T_n = T_0$, and clearly $L_0$
is contained in $L_n$ because every prime of $T_n$ above a prime in $S$ must also split completely in the extension $L_0T_n/T_n$. Thus $Y/Z = Y_0$.
	 Hence, by the snake lemma, for every integer $m \geq 0$,  we have an isomorphism 
	\[ Y_0/p^mY_0 \cong Y/({p^mY+Z}).    \]
	Let $\gamma$ be a generator of $\CG$. For $r \geq 1$, define $\nu_r= (\gamma^{p^r}-1)/(\gamma-1)$, and note that $\nu_r$ is an element of the maximal ideal of the local ring $\BZ_p[\CG]$. We shall prove later that 
	\begin{equation} \label{eq: iwasawa}
	Y_r= Y/ Z^{\nu_r} \text{ for each }  1\leq r \leq n.
	\end{equation}
	whence, again by the snake lemma, we have
	\begin{equation}\label{I2}
	 Y_r/p^m Y_r \cong Y/({p^m Y + Z^{\nu_r} }).
	 \end{equation}
	Let us assume for the moment that \eqref{eq: iwasawa} is valid. In view of this last equation, our  hypothesis that $\#(A_0/p^mA_0)=\#(A_1/p^m A_1)$ implies that 
	\begin{equation}\label{I1}
	p^mY+Z= p^mY+Z^{\nu_1}.
	\end{equation}		
	Put $W = (p^mY +Z)/p^mY$. Then $\nu_1 (W) =( p^mY+Z^{\nu_1})/p^mY$.  But, in view of \eqref{I1}, we have $(p^mY+Z^{\nu_1})/p^mY = W$, and so $\nu_1(W) = W$. Since $\nu_1$ lies in the maximal ideal of the local ring $\BZ_p[\CG]$, It follows from Nakayama's lemma that $ W=0$, whence $Z \subset p^mY$. 
	It now follows from \eqref{I2} that
	\[A_r/p^mA_r \cong Y_r/p^mY_r \cong Y/p^mY \text{ for any } 0\leq r \leq n.   \]
	Thus, to complete the proof of Proposition \ref{prop: stable}, it remains to prove \eqref{eq: iwasawa}. Let $G=\Gal(L/T_0)$, where we recall that $L$ is the maximal unramified abelian $p$-extension of $T_n$ such that all primes of $T_n$, which lie above a prime in $S$, split completely in $L$. Let $S'$ be the subset of $S$ consisting of those primes which are inert or ramified (hence totally ramified by our condition) in $T_n$ and write $S'=\{\fp_1, \cdots, \fp_s\}$. Let $\fp_{s+1},\cdots, \fp_t$ be the primes of $T_0$ which are ramified (hence totally ramified) in $T_n$ but not in $S'$. 
	Let $\tilde{\fp}_j$ ($1\leq j \leq t$) be some prime of $L$ lying above $\fp_j$ and write $D_j$ (resp. $I_j$) for the decomposition (resp. inertia) subgroup of $\tilde{\fp}_j$ in $G$. We let 
	\[  E_j = \begin{cases}  D_j & \text{ if } 1\leq j \leq s, \\
	I_j  & \text{ if } s+1 \leq j \leq t
	\end{cases}  \]
	It follows from our definition of $L$ that, for each $1\leq j \leq t$, we have $E_j\cap Y=0$, and $E_j$ is naturally isomorphic to $G/Y=\CG$.
	Hence each $E_j$ is cyclic, and we can choose a generator $\sigma_j$ of $E_j$ for $1\leq j \leq t$ such that the coset of $\sigma_j$ in $G/ Y$ is $\gamma$. Put $y_j=\sigma_j \sigma^{-1}_1$ for $2\leq j \leq t$, so that $y_j$ is in $Y$. Now suppose $\tilde{\fp}'_j$ is another prime of $L$ lying above $\fp_j$. Then the decomposition group of $\tilde{\fp}'_j$ for the extension $L/T_0$ is conjugate to $E_j$ whence it is contained in $\langle  E_j ,[G,G]	\rangle$, where $[G,G]$ is the commutator subgroup of $G$. It follows that $L_0$, being an intermediate field of $L/T_0$, is fixed by $\langle [G,G], E_1, \cdots, E_t\rangle = \langle [G,G], \sigma_1, y_2, \cdots, y_t\rangle$. We claim that
	\begin{equation}\label{eq: commutator}
	[G,G] = Y^{\gamma-1}.
	\end{equation}
	Indeed, $Y^{\gamma-1}$ is clearly contained in $[G,G]$ and it is normal subgroup of $G$ by a direct computation. Thus, in order to show \eqref{eq: commutator}, it suffices to show that $G/(Y^{\gamma-1})$ is abelian. Any element of $G/(Y^{\gamma-1})$ has the form $\tilde{\gamma}^k y \bmod Y^{\gamma-1}$ for some integer $k$, where $\tilde{\gamma}\in G$ is a lifting of $\gamma$ and $y\in Y$. Then \[ \tilde{\gamma}^{k_1} y_1 \tilde{\gamma}^{k_2} y_2 \equiv \tilde{\gamma}^{k_1+k_2} y_1 y_2 \bmod  Y^{\gamma-1}\] where $k_1,k_2 \in \BZ$ and $y_1, y_2 \in Y$. 
	It follows that $G/(Y^{\gamma-1})$ is indeed abelian. This proves the assertion \eqref{eq: commutator}. Hence we have
	\[ \Gal(L/L_0)=\langle  Y^{\gamma-1}, \sigma_1, y_2,\cdots, y_t \rangle.\]
	Since  $Y$ is a normal subgroup of $G$, we have  $Y\cap \Gal(L/L_0)= \langle  Y^{\gamma-1}, y_2,\cdots, y_t \rangle.$ Recall that $G=E_1 Y$. Thus the inclusion $Y\hookrightarrow G$ induces the following isomorphism
	\begin{equation}\label{eq: unram quotient}
	Y/ \langle Y^{\gamma-1}, y_2,\cdots, y_t \rangle \cong G/ \Gal(L/L_0)=Y_0  . 
	\end{equation}
	But $Y_0=Y/Z$ where $Z$ is defined in \eqref{eq: Z}, because there is at least one totally ramified prime in the extension $T_n/T_0$ . Therefore, we have
	\[ Z=\langle Y^{\gamma-1}, y_2,\cdots, y_t \rangle.\] Now one can repeat the above arguments  proving \eqref{eq: unram quotient} for the extension  $L/K_r$ for $r=1,..., n$. To do this, we must replace $\gamma$ by $\gamma^{p^r}$, and $\sigma_j$ by  $\sigma^{p^r}_j$, and work with the decomposition group generated by $\sigma^{p^r}_j$. We then find
	\[ Y_r = Y/ \langle Y^{\gamma^{p^r}-1}, \sigma^{p^r}_2 \sigma^{-p^r}_1,\cdots, \sigma^{p^r}_t \sigma^{-p^r}_1 \rangle.    \]
	Note that $Y^{\gamma^{p^r} -1}=(Y^{\gamma-1})^{\nu_r} $. Hence the assertion \eqref{eq: iwasawa} will follow, and the proof of Proposition \ref{prop: stable} will be complete, once we have established the following lemma.  
\end{proof}
 
 \begin{lem} For $r=0,\ldots,n$ we have,
 $$
 \sigma^{p^r}_j \sigma^{-p^r}_1=  y^{\nu_r}_j  \, \, \, (j = 2,..., t),
 $$
where  $\nu_r= (\gamma^{p^r}-1)/(\gamma-1)$.
\end{lem}
\begin{proof} Recall that, by definition, $y^{\gamma}=\tilde{\gamma} y \tilde{\gamma}^{-1}$ for any lifting $\tilde{\gamma}\in G$ of $\gamma$, and any $y\in Y$; and by definition, $\sigma_j$ ($1\leq j\leq t$) is a lifting of $\gamma$.
 We shall prove by induction that, for any integer $k\geq 1$, we have
 \begin{equation}\label{i3}
  \sigma^{k}_j\sigma^{-k}_1 =  y^{  \frac{\gamma^{k}-1}{\gamma-1}}_j.
  \end{equation} 
When $k=1$, this is just the defining equation of $y_j$. Suppose $k \geq 2$, and that \eqref{i3} holds for $k-1$. Then \[ \sigma^{k}_j\sigma^{-k}_1 = \sigma_j  y^{  \frac{\gamma^{k-1}-1}{\gamma-1}}_j  \sigma^{-1}_1 =  \sigma_j  y^{  \frac{\gamma^{k-1}-1}{\gamma-1}}_j   \sigma^{-1}_j \sigma_j \sigma^{-1}_1= y^{ \gamma \left(\frac{\gamma^{k-1}-1}{\gamma-1}\right)+1}_j  =y^{  \frac{\gamma^k-1}{\gamma-1}}_j ,    \]
completing the proof by induction.
\end{proof}

\bigskip
	
\begin{proof}[Proof of Proposition~\ref{prop: bound clgp}]
	We use the notation of the proof of Proposition~\ref{prop: stable}. Let $Y'$ be the $\BZ_p[\CG]$-module submodule of $Y$ generated by $y_2,\cdots,y_t$. It follows from \eqref{eq: unram quotient} that 
	\[(Y/Y')/(Y/Y')^{\gamma-1} = Y_0.  \]
	Let $\fm=(p,\gamma-1)$, so that $\fm$ is the maximal ideal of $\BZ_p[\CG]$. Thus 
	\[(Y/Y')/{\fm (Y/Y')}=Y_0/{pY_0}.\]
	Hence, by Nakayama's lemma, $Y/Y'$, as a $\BZ_p[\CG]$-module, can be generated by $g_0$ elements, where we recall that $g_0$ is equal to the $\BF_p$-dimension of $A/{pA}$. Therefore, $Y$ can be generated by $g_0+t-1$ elements as a $\BZ_p[\CG]$-module. Since $\mathrm{rk}_p \BZ_p[\CG]=p^n$, we have $\mathrm{rk}_p(Y)\leq p^n(g_0+t-1)$.
	The last assertion of Proposition~\ref{prop: bound clgp} follows from the fact that $A_n$ is a $\BZ_p[\CG]/(N_\CG)$-module when $N_\CG A_n=0$, and that $\mathrm{rk}_p (\BZ_p[\CG]/(N_\CG))=p^n-1$.  \end{proof}

We now return to the discussion of the preliminary ideas behind the proof of Theorem \ref{thm: Y}.  The first important fact is the following. Note that the definition of the fields $K$, $F$, $F'$ and $J$ remains valid for all primes $q\equiv 3\bmod 4$, and the following result holds for all such primes. Of course, it is classical that $K$ has odd class number, and it is already shown in \cite{CL}, Theorem 3.8 that $F$ has odd class number for all primes $q \equiv 7 \mod 8.$

\begin{prop}\label{prop: clno1}
Assume $q$ is any prime such that $q\equiv 3\bmod 4$. Then all of the fields $K, F, F', D$ and $J$ have odd class number.
\end{prop}

\begin{proof}
    Let $i=\sqrt{-1}$. Of course, $\BQ(i)$ has class number $1$. 
    Note that $q$ is inert in $\BQ(i)$, and that the prime $q\BZ[i]$ is totally ramified in $J$ and in $D$. Suppose $q\equiv 7\bmod 8$. Let $v$ be a prime of $J$ lying above $2$. Then it follows from Lemma~\ref{lem: ram F/K} that $J_v$ is equal to $\BQ_2(\sqrt{-1},\sqrt{-3})$ when $q\equiv 7\bmod 16$ and to $\BQ_2(\sqrt{-1})$ if $q\equiv 15\bmod 16$. Thus the ramification index of $v$ in $J/\BQ$ is $2$. Note that $(1+i)\BZ[i]$ is the prime of $\BQ(i)$ above $2$ which is ramified in $\BQ(i)/\BQ$. It follows that $(1+i)\BZ[i]$ is unramified in $J$. Thus, we conclude that there is only one prime ramified in $D/\BQ(i)$. It follows from Proposition~\ref{prop: bound clgp} that $D$ has odd class number. Applying Proposition~\ref{prop: bound clgp} on $J/D$ implies that $J$ has odd class number. Since $J/F$ (resp. $J/F'$) is ramified at $w^*$ (resp. $w$), it follows from class field theory that $F$ and $F'$ both have odd class numbers.

	 Assume next that $q\equiv 3\bmod 8$. 
	 We shall apply Chevalley's formula~\eqref{eq: amb} to the quartic cyclic extension $J/\BQ(i)$. Let $v$ be a prime of $J$ above $2$. Then $J_v = \BQ_2(\sqrt[4]{-3},i)$ or $\BQ_2(\sqrt[4]{-11},i)$, according as $q\equiv 3\bmod 16$ or $q\equiv 11\bmod 16$. We leave the reader to check that neither $\BQ_2(\sqrt[4]{-3})$ nor $\BQ_2(\sqrt[4]{-11})$ is Galois over $\BQ_2$. It follows that the maximal abelian extension of $\BQ_2$ inside $J_v$ is $\BQ_2(\sqrt{-3},i)$. In particular, we conclude that the maximal unramified extension of $\BQ_2$ inside $J_v$ is $\BQ_2(\sqrt{-3})$. This proves that the ramification index and the inertial degree of $(1+i)\BZ[i]$ in $J$ are both equal to $2$. Since $q\equiv 3\bmod 8$, the quartic Hilbert symbol 
	\[ \hilbert{ i}{-q}{q\BZ[i]}_4 \equiv i^{\frac{q^2-1}{4}} =-1 \bmod q\BZ[i]. \]
	It follows that $i$ is not a norm from $J$, whence the unit index in Chevalley's formula for $J/\BQ(i)$ is divisible by $2$. Thus, in this case, Chevalley's formula when $S = \phi$ tells us the $\Gal(J/\BQ(i))$-invariant of $\Cl_J$ is $0$ where $\Cl_J$ is the class group of $J$, whence, by Nakayama's lemma, the class number of $J$ is odd.  We just proved that the unique prime of $J$ above $2$ is totally ramified in $J/K$. It follows that $D,F $ and $F'$ must also have odd class number. This completes the proof of Proposition~\ref{prop: clno1}. \end{proof}
	
\medskip
If $R$ denotes any of our four number fields $D, F, F', J$, we recall that $R_n = RK_n$ and $R_n^* = RK_n^*$ for $1 \leq n \leq \infty$.
Recall also that we have chosen our notation so that  the prime factor $\fp$
(resp. $\fp^*$) of $2\CO_K$ is ramified in $F$ (resp. $F'$).

\begin{prop}\label{prop: clgp Z_p-tower}
Assume $q\equiv 7\bmod 8$. Then the class numbers of $D_n$, $D^*_n$, $F_n$ and $F'^*_n$ are all odd for all $n\geq 0$. Moreover, for all $n \geq 1$, the class numbers of $F^*_n$, $F'_n$, $J_n$ and $J^*_n$ are all odd if $q\equiv 7\bmod 16$, and are all even if $q\equiv 15\bmod 16$.
\end{prop}

\begin{proof}
Note that there is only prime of $D$ (resp. $F$) above $\fp$ and this prime must be totally ramified in $D_\infty$ (resp. $F$)  since $D$ (resp. $F$) has odd class number by Proposition~\ref{prop: clno1}. It follows by a classical argument in Iwasawa theory (see the proof of Theorem 3.8 of \cite{CL}) that $D_n$ (resp. $F_n$) has odd class number for all $n \geq 0$. Alternatively, one can use Proposition~\ref{prop: bound clgp} or Chevalley's formula to show $D_n$ (resp. $F_n$) also has odd class number for each $n\geq 1$. As in Lemma~\ref{lem: symmetry}, we also conclude that $D^*_n$ and $F'^*_n$ have odd class numbers too. 

\medskip

Assume now that $q\equiv 7\bmod 16$. 
Then, by Lemma \ref{lem: ram F/K}, there is exactly one prime of $v$ of $J$ lying above $\fp$, and exactly one prime $v^*$ of $J$  above $\fp^*$. Moreover, $v$ (resp. $v^*$) is then the unique ramified prime in $J_\infty$ (resp. $J_\infty^*)$, and it must be totally ramified since $J$ has odd class number by Proposition~\ref{prop: clno1}. It now follows either by a classical argument in Iwasawa theory, or again by using  Proposition~\ref{prop: bound clgp} that $J_n$ (resp. $J_n^*$) has odd class number for all $n \geq 1$. Also $J_n/F'_n$ is totally ramified at the unique prime of $F'_n$ above $\fp$, and $J_n^*/F_n^*$ is totally ramified at the unique prime of $F_n^*$ above $\fp^*$, whence $F'_n$, $F_n^*$ also have odd class number, completing the proof when $q\equiv 7\bmod 16$.

\medskip

Suppose next that $q\equiv 15\bmod 16$. We first prove that  $F^*_1 = FK^*_1$ has even class number. We shall apply Chevalley's formula \eqref{eq: amb} to the quadratic extension $F^*_1/F$. Note that, by Lemma \ref {lem: ram F/K}, $\fp^*$ splits in $F$, say $\fp^*\CO_F = w^*_1 w^*_2$, where $\CO_F$ denotes the ring of integer of $F$. Since $\fp^*$ is ramified in $K^*_1$, both $w^*_1$ and $w^*_2$ are ramified in the extension $F^*_1/F$, and no other primes of $F$ are ramified. We now show that each unit in $\CO_F$ is a norm of $F^*_1$. Recall that $h_F$ denotes the class number of $F$.   Let $\gamma$ be a generator of $w^{h_F}$, and define $\beta = N_{F/K}(\gamma)$. Then, since $w$ is the unique prime of $F$ above $\fp$, we have $\beta \CO_F = \gamma^2 \CO_F$. Thus $\gamma^2/\beta=\pm \eta^k$ for some integer $k$, where we recall that $\eta$ denotes a fundamental unit of $F$. Again, as in the proof of Proposition \ref{prop: log lower bound}, the integer $k$ must be odd, since otherwise the extension $F/K$ would be unramified outside the set of primes dividing 2. Hence, to prove that each unit in $\CO_F$ is a norm from $F_1^*$, it suffices to show that both $\gamma^2/\beta$ and $-1$  lie in $N_{F^*_1/F}(F^*_1)$. 
Let $\pi^*$ be a generator of ${\fp^*}^t$ where $t$ is the order of the ideal class of $\fp^*$ in the class group of $K$. Then it follows from the proof of Proposition~\ref{prop: K_1} that, by adjusting the sign of $\pi^*$ if necessary, we have $K^*_1 = K(\sqrt{\pi^*})$, whence $F^*_1=F(\sqrt{\pi^*})$.  Moreover, since $\fp$ splits in $K^*_1$ by Proposition~\ref{prop: K_1}, we must have $\pi^*\equiv 1 \bmod \fp^3$, or equivalently $\pi^*$ is a square in $K_\fp=F_{w^*_1}=F_{w^*_2}$. Thus, for $j=1,2$ we have the following equalities of quadratic Hilbert symbols:
\[ \hilbert{ \pm \gamma^2/\beta}{\pi^*}{ w^*_j} = \hilbert{\pm \beta}{\pi^*}{ w^*_j}= \hilbert{\pm \beta}{\pi^*}{ \fp^* } = \hilbert{\pm\beta}{\pi^*}{\fp}=1. \]
Indeed, the first equality is obvious; the second equality holds because $\beta, \pi^* \in K_{\fp^*}=F_{w^*_j}$; the third equality is by the product formula for Hilbert symbols and the fact that the Hilbert symbol of $\beta$ and $\pi^*$ at a prime outside $\fp,\fp^*$ is trivial by local class field theory; the last equality is thanks to the fact proven above that $\pi^*$ is a square in $K_\fp$.  Now every unit in $\CO_F$ is clearly a local norm at the unramified primes of $F$ in $F^*_1$.  Thus it follows from the above computation and Hasse's norm theorem that every unit of $F$ is a global norm from $F_1$. By Chevalley's formula \eqref{eq: amb} applied to the extension $F_1^*/F$, we conclude that the class number of $F^*_1$ must be even. Thus, since the primes of $F$ lying above $\fp^*$ are totally ramified in $F^*_n$, it follows that $F^*_n$ has even class number for all $n \geq 1$. By interchanging the primes $\fp$ and $\fp^*$, the same argument proves $F'_n$ has even class number for all $n \geq 1$.

We claim that the prime $w^*_i$ ($i=1,2$) of $F$ is totally ramified in $J^*_1$ whence it will be totally ramified in $J^*_\infty$. Indeed, since $J$ has odd class number by Proposition~\ref{prop: clno1}, we must have that $J^*_1/J$ is ramified at some prime of $J$ lying above $w^*_i$. As $w^*_i$ is ramified in $J/F$, this proves the claim. It follows that $J^*_1/F^*_1$ is totally ramified at the primes above $w^*_i$ ($i=1,2$) whence so is $J^*_n/F^*_n$. Thus $J^*_n$ has even class number since we just proved that $F^*_n$ has even class number. Using the argument of Lemma~\ref{lem: symmetry} shows that $J_n$ has even class number. \end{proof}

 Apart from some brief comments made about the case $q \equiv 15 \mod 16$, at the very end of this section, we shall assume from now on that $q\equiv 7\bmod 16$. We assume that $w$ lying above $\fp$ and $w^*$ lying above $\fp^*$ are the primes of $F$ as in Theorem \ref{thm: log}.
 Thus $F_{w}=\BQ_2(\sqrt{3})$ and $F_{w^*}=\BQ_2(\sqrt{-3})$. Moreover $w$ is inert in $J$ and $w^*$ is ramified in $J$.
 If $M$ is a number field, we let $M^{\cyc}_\infty$ be the cyclotomic $\BZ_2$-extension of $M$ and write $M^{\cyc}_{n}$ for its $n$-th layer. By class field theory, the compositum of all $\BZ_2$-extension $\CK_\infty$ of $K$ is $K_\infty K^{\cyc}_\infty$, which is also equal to $K^*_\infty K^{\cyc}_\infty $, and to $K_\infty K^*_\infty$. Finally, let $\CF_n=FK_nK^*_n$ and $\CJ_n=JK_nK^*_n$ for each $n$.

\begin{prop}\label{prop: S-clgp}
	Assume $q\equiv 7\bmod 16$. Then, for each $n\geq 0$, the $2$-primary subgroup of $\Cl_{F^{\cyc}_n,S}$ is zero; here $S= \{w\}$ is the set consisting of the unique prime of $F$ lying above $\fp$. Moreover, the class number of $\CF_n = FK_nK_n^*$ is odd for every $n\geq 0$.
\end{prop}

\begin{proof}	
Both of the primes $w$ and $w^*$ are ramified in $F^{\cyc}_1=F(\sqrt{2})$, since $F_w(\sqrt{2})$ and $F_{w^*}(\sqrt{2})$ are ramified extension of $F_w$ and $F_{w^*}$, respectively. Hence $w$ and $w^*$ must be totally ramified in the $\BZ_2$-extension $F^{\cyc}_\infty/F$. Let $w^{\cyc}_n$ denote the unique prime of $F^{\cyc}_n$ above $w$ for each $n\geq 1$.
Write $\Cl_{F^{\cyc}_n,S}(2)$ for the $2$-primary subgroup of $\Cl_{F^{\cyc}_n,S}$. Since $F$ has odd class number,  and $w$, $w^*$ are the only primes of $F$ ramified in $F^{\cyc}_1$, it follows from the final assertion of Proposition~\ref{prop: bound clgp} that $\Cl_{F^{\cyc}_1}(2)$ is a cyclic group. Now  $F^{\cyc}_1 K_1K^*_1$ is an unramified extension of $F^{\cyc}_1$, because, on the one hand, it is unramified outside the primes above $\fp$, and, on the other hand, it is also unramified outside the primes above $\fp^*$ since $F^{\cyc}_1 K_1=F^{\cyc}_1K^*_1$. 
The prime $w^{\cyc}_1$ is inert in the extension $F^{\cyc}_1 K_1/F^{\cyc}_1$ by Proposition~\ref{prop: K_1},  and so, by class field theory, it must be inert in the $2$-Hilbert class field $L(F^{\cyc}_1)$ of $F^{\cyc}_1$ since we have shown that the extension $L(F^{\cyc}_1)/F^{\cyc}_1$ is cyclic.
Hence, by class field theory, the whole group $\Cl_{F^{\cyc}_1}(2)$ must be equal to the $2$-primary part of the subgroup generated by the ideal class of $w^{\cyc}_1$. In other words, $\Cl_{F^{\cyc}_1,S}(2)=0$. Therefore, applying Proposition~\ref{prop: stable in Z_p}, we conclude that the same assertion holds for every layer $F^{\cyc}_n$ of the cyclotomic 
$\BZ_2$-extension of $F$. For the second assertion of the Proposition, we note that $(w^{\cyc}_n)^{2^n}=w\CO_{F^{\cyc}_n}$. Hence, since $F$ has odd class number, it follows from the first assertion of the Proposition that the $2$-primary part of $\Cl_{F^{\cyc}_n}$ has order dividing $2^n$. But $\CF_n/F^{\cyc}_n$ is already an unramified cyclic extension of degree $2^n$. This shows that $L(F^{\cyc}_n)=\CF_n$. Now $Y(\CF_n) = \Gal(L(\CF_n)/\CF_n)$ is endowed with the usual action of $\fd_n = \Gal(\CF_n/F^{\cyc}_n)$, and we then have $Y(\CF_n)_{\fd_n} = 0$ since $L(F^{\cyc}_n)=\CF_n$. Hence, by Nakayama's lemma, $Y(\CF_n) = 0$, proving that the class number of $\CF_n$ is odd. This completes the proof.
\end{proof}

\begin{lem}\label{lem: unit index}
Assume $q\equiv 7\bmod 16$. Then the index $[ \CO^\times_{\CF_1} :  \CO^\times_{\CF_1} \cap N_{\CJ_1/\CF_1}(\CJ^\times_1) ]$ is divisible by $2$.
\end{lem}

\begin{proof}
	Since $F_1=FK_1$ is a bi-quadratic extension of $K$, it has three nontrivial intermediate fields. Clearly $F$ and $K_1$ are two of them, and we denote the third one by $E$. Since the prime $\fp$ of $K$ is totally ramified in $F_1$, and $F_1$ has odd class number by Proposition \ref{prop: clgp Z_p-tower}, it follows that also $E$ has odd class number.
We first show that there exists a unit $u\in \CO^\times_{F_1}$ such that $N_{F_1/E}(u)=-1$, where $\CO^\times_{F_1}$ denotes the unit group of $F_1$.  Now the only primes of $K$ which ramify in $F_1$ are the primes $\fp$ and $\fq = \sqrt{-q}\CO_K$, with respective ramification indices $4$ and $2$. Moreover, $\fq$ must ramify in $E$, because $F_1 = K_1E$ and $\fq$ does not ramify in
$K_1$. Hence the extension $F_1/E$ is ramified at the unique prime of $E$ lying above $\fp$. Applying Chevalley's formula \eqref{eq: amb} to the extension $F_1/E$
with $S= \phi$, and recalling that $F_1$ and $E$ have odd class number, we conclude that the index $[\CO^\times_{E}:\CO^\times_{E}\cap N_{F_1/E}({F_1}^\times)]$ is odd;
 here $\CO^\times_{E}$ denotes the group of units of $E$. Hence there exists $\beta \in F_1$ such that $ N_{F_1/E}(\beta) = -1$. Let $\beta \CO_{F_1} = \fb$. Since the norm of $\fb$
 from $F_1$ to $E$ is the trivial ideal, it is clear that we must have $\fb = \fc/\fc^\tau$ for some ideal $\fc$ of $F_1$; here $\tau$ denotes the non-trivial element of $\Gal(F_1/E)$.
But, since $F_1$ has odd class number, we must have $\fc^k = \theta \CO_{F_1}$ for some non-zero $\theta \in F_1$ and some odd integer $k \geq 1$. We then have $\beta^k = u(\theta/\theta^\tau)^k$ for some
$u \in \CO^\times_{F_1}$, and, since $k$ is odd, it follows that $-1= N_{F_1/E}(u)$, proving our first claim. 
	Our second step is to show that there exists a unit $u'\in \CO^\times_{\CF_1}$ such that $N_{\CF_1/F_1}(u')=u$, where $\CO^\times_{\CF_1}$ is the group of units of $\CF_1$.
	Note that $w^*$ splits into two primes of $F_1$ and these two are precisely the ramified primes of $F_1$ in $\CF_1$. Let $v$ denote any one of these two primes of $F_1$, let $v'$ denote the prime of $\CF_1$ above $v$ and let $v_0$ be the prime of $E$ below $v$. We have $(F_1)_{v} \cong \BQ_2(\sqrt{-3})$ and $E_{v_0}\cong \BQ_2$. It follows from the norm functoriality of Hilbert symbols that 
	\[\hilbert{ u}{2}{v}= \hilbert{-1}{2}{v_0} = 1.\]
	It follows that the global unit $u$ is a local norm at any prime of $F_1$ as as $\CF_1 = F_1 F^{\cyc}_1=F_1(\sqrt{2})$. Thus, by Hasse's norm theorem, we have $ u \in N_{\CF_1/F_1}(\CF^\times_1)$.
	Since $\CF_1$ has odd class number, an entirely similar argument to that given above shows that there exists an odd integer $m$ such that $u^m$ is in $N_{\CF_1/F_1}\CO^\times_{\CF_1}$. But since an even power of $u$ is clearly the norm of a unit in  $\CO^\times_{\CF_1}$, we conclude that
	\[ u =N_{\CF_1/F_1}(u') \text{ for some unit }u'\in   \CO^\times_{\CF_1} .\]
	Now we claim 
	\[
	u' \notin N_{\CJ_1/\CF_1}(\CJ^\times_1).  
	\]
	As $\CJ_1=\CF_1(\sqrt{-1})$, this claim follows from the following values of Hilbert symbols:
	\[\hilbert{u'}{-1}{v'} = \hilbert{u}{-1}{v}=\hilbert{-1}{-1}{v_0}=-1.   \]
	Thus the unit index in the lemma is indeed nontrivial, and the proof is complete. \end{proof}

\begin{prop}\label{prop: clgp CJ_1} 
Assume $q\equiv 7\bmod 16$. Then the class number of $\CJ_1$ is odd.
\end{prop}

\begin{proof}
	The primes of $\CF_1$ ramified in $\CJ_1$ are precisely the two primes above $w^*$. Applying Chevalley's formula \eqref{eq: amb} on $\CJ_1/\CF_1$, it follows from Lemma~\ref{lem: unit index} that the class number of $\CJ_1$ is odd by Nakayama's lemma.
\end{proof}

\bigskip

\begin{proof}[Proof of Theorem~\ref{thm: Y}]
	Since $q \equiv 7 \mod 16$, there is exactly one prime of $J$ lying above $\fp$ (resp. $\fp^*$) , and we denote this prime by $v$ (resp. $v^*$). Then $v$ (resp. $v^*$) is totally ramified in $J_n = JK_n$ (resp. $J^*_n = JK_n^*$)  and unramified in $J^*_n $ (resp. $J_n$) for all $n \geq 0$. It follows that the ramification index of $v$ in $\CJ_n = JK_nK_n^*$ is exactly $2^n$. Thus,
	for any $n\geq 0$, the primes of $J_n = JK_n$ lying above $v^*$ are totally ramified in $\CJ_n$ whence they are totally ramified in the  $\BZ_2$-extension $J_n J^*_\infty/J_n$. Clearly, $J_n J^*_\infty/J_n$ is unramified outside the primes above $v^*$. Hence the $\BZ_2$-extension $J_nJ^*_\infty/J_n$ is totally ramified at each ramified prime for all $n\geq 0$. Similarly, the $\BZ_2$-extension $J^*_nJ_\infty/J^*_n$ is also totally ramified at each ramified prime for all $n\geq 0$.

	 Now the class numbers of $J_1$ and $\CJ_1= J_1J_1^*$ are both odd by Proposition~\ref{prop: clgp Z_p-tower} and Proposition~\ref{prop: clgp CJ_1}. Applying Proposition~\ref{prop: stable in Z_p} to the $\BZ_2$-extension $J_1 J^*_\infty/J_1$, we conclude that $J_1J^*_n$ has odd class number for all $n\geq 0$. Then, the class number of $J^*_n$ is also odd, since $J_1J^*_n/J^*_n$ is totally ramified at the primes above $v$. Now,
	 applying Proposition~\ref{prop: stable in Z_p} on the $\BZ_2$-extension $J^*_nJ_\infty/J^*_n$, we conclude that the class number of $J^*_n J_r$ is odd for every $r\geq 0$. In particular, $\CJ_n$ has odd class number for every $n\geq 0$. But every finite unramified abelian $2$-extension of $\CJ_\infty$ is clearly defined by an equation with coefficients lying in $\CJ_n$ for some $n \geq 0$, and thus $L(\CJ_\infty)= \cup_{n} L(\CJ_n)$. Hence $Y(\CJ_\infty)=0$. Finally, since the extensions $\CJ_\infty/\CD_\infty$, $\CJ_\infty/\CF_\infty$, $\CJ_\infty/\CF'_\infty$ and $\CF_\infty/\CK_\infty$ are all ramified quadratic extensions, it follows from class field theory that $Y(\CD_\infty)=Y(\CF_\infty)=Y(\CF'_\infty)=Y(\CK_\infty)=0$, and the proof is complete.
\end{proof}

 Assume $q\equiv 7\bmod 16$. By a similar argument to that used in Proposition~\ref{prop: S-clgp}, one can show that, for each $n$, $L(J^{\cyc}_n)=\CJ_n$ whence the $2$-primary part $\Cl_{J^{\cyc}_n}(2)$ of the class group of $J^{\cyc}_n$ is cyclic of order $2^n$. But unlike the result in Proposition~\ref{prop: S-clgp}, $\Cl_{J^{\cyc}_n}(2)$ is not generated by the primes above $w$ and $w^*$, since the primes of $J^{\cyc}_n$ above $w$ or $w^*$ all split in $J^{\cyc}_n K_1 = J^{\cyc}_nK^*_1$. This result on $\Cl_{J^{\cyc}_n}(2)$ extends a joint work of the second author \cite{L2}.

\medskip
 
Kida \cite{Kid} and Ferrero \cite{Fer} have independently computed the unramified Iwasawa module of the cyclotomic $\BZ_2$-extension of an imaginary quadratic field. 
For our field $K$, their formula gives $Y(K^{\cyc}_\infty)\cong \BZ^\lambda_2$ where $\lambda = -1 + 2^{(-3+ \ord_{2}(q^2-1))}$. It follows that $Y(\CK_\infty)$ has a quotient whose $\BZ_2$-rank is $\lambda-1$, since $\CK_\infty/K^{\cyc}_\infty$ is an unramified $\BZ_2$-extension. Note that $\lambda -1 \geq 1$ when $q\equiv 15\bmod 16$ and it is unbounded when $q$ varies in the family of primes $\equiv 15\bmod 16$. This certainly implies the same assertion holds for $Y(\CD_\infty)$, $Y(\CF_\infty)$ and $Y(\CJ_\infty)$.

\section{Relationship to the arithmetic of a family of abelian varieties}

We now discuss the connexion of Theorem \ref{mir} with the arithmetic of a certain abelian $B$ variety with complex multiplication, and use it to prove Theorem \ref{mirB}. As in the Introduction, let $H = j(\CO_K)$ be the Hilbert class field of $K$.  Let $A$ be the unique elliptic curve defined over $\BQ(j(\CO_K))$ with complex multiplication by $\CO_K$,
minimal discriminant $(-q^3)$, and which is isogenous to all of its conjugates over $H$ (see \cite{Gross1}, Chap. 3). Let $B/K$ be the $h$-dimensional abelian variety which
is the restriction of scalars from $H$ to $K$ of $A$ (see \cite{Gross1}, Chap. 4); here we recall that $h$ is the class number of $K$ . Let $\msb = \End_K(B)$, and  $\mst = \msb \otimes \BQ$, so that $\mst$ is a CM field of degree $h$ over $K$. As was explained in the Introduction, the fact that $h$ is odd implies that the primes $\fp$, $\fp^*$ of $K$
are both unramified in $\mst$, and the fact that the torsion subgroup of $B(K)$ is $\CO_K/2\CO_K$ implies that there exist unramified primes $\fP, \fP^*$ of inertial degree 1
of $\mst$ lying above $\fp, \fp^*$, respectively (see \cite{Gross1}, \cite{BG}, \cite{CL}) for detailed proofs of these facts). For $1 \leq n \leq \infty$, we write $B_{\fP^n}$ (resp. $B_{\fP^{*n}}$) for the Galois module of $\fP^n$ (resp. $\fP^{*n}$) -division points on $B$, both of which are isomorphic as abelian groups to $\BZ/2^n\BZ$ (resp. $\BQ_2/\BZ_2$) when $n$ is finite (resp when $n = \infty$), but have a highly non-trivial Galois action on them when $n \geq 2$.  It is shown in \cite{CL} (see Theorem 2.4  and Lemma 7.11), that, always with our standard notational convention, we have
\begin{equation}\label{f1}
 F= K(B_{\fP^2}), \, \, F_\infty = K(B_{\fP^\infty}), \, \, F' = K(B_{\fP^{*2}}), \, \, F'^*_\infty = K(B_{\fP^{*\infty}}).
\end{equation}
Moreover, the Weil pairing gives rise to an isomorphism of Galois modules $B_{\fP^{*2}} \cong \Hom (B_{\fP^2}, \mu_4)$, where $\mu_4$ denotes the group of $4$-th roots of unity, whence we obtain 
\begin{equation}\label{f2}
J = K(B_{\fP^2}, B_{\fP^{*2}}), \, \,  J_\infty = J(B_{\fP^\infty}), \, \, J^*_\infty = J(B_{\fP^{*\infty}}).
\end{equation}
For any algebraic extension $\fF$ of $K$,  we recall that the $\fP^\infty$-Selmer group of $B/\fF$ is defined by
\begin{equation}\label{f3}
\Sel_{\fP^\infty} (B/\fF) = \Ker (H^1(\fF, B_{\fP^\infty}) \to \prod_{v}H^1(\fF_v, B)),
\end{equation}
where $v$ runs over all non-archimedean places of $\fF$, and, as usual, $\fF_v$ denotes the union of the completions at $v$ of the finite extensions of $K$ contained in $\fF$,
and there is an entirely analogous definition for $\Sel_{\fP^{*\infty}} (B/\fF)$. Now it is shown in \cite{CL}, Theorem 2.4, that $B$ has good reduction everywhere over the fields 
$F$ and $F'$, from which, in view of \eqref{f1} and \eqref{f2}, it follows easily that  
\begin{equation}\label{f4}
\Sel_{\fP^\infty} (B/F_\infty) = \Hom(X(F_\infty), B_{\fP^\infty}), \, \Sel_{\fP^\infty} (B/J_\infty) = \Hom(X(J_\infty), B_{\fP^\infty}),
\end{equation}
and 
\begin{equation}\label{f5}
\Sel_{\fP^{*\infty}} (B/F'^*_\infty) = \Hom(X^*(F'^*_\infty), B_{\fP^{*\infty}}), \, \Sel_{\fP^{*\infty}} (B/J^*_\infty) = \Hom(X^*(J^*_\infty), B_{\fP^{*\infty}}).
\end{equation}
When $q \equiv 7 \mod 16$, Theorem 3.1 of \cite{CL} shows that $X(F_\infty) = X*(F'^*_\infty) = 0$, whence
\begin{equation}\label{f6}
\Sel_{\fP^\infty} (B/F_\infty) = \Sel_{\fP^{*\infty}} (B/F'^*_\infty) = 0.
\end{equation}
Further, Theorem \ref{mir} is then clearly equivalent to the following.

\begin{thm}\label{f7} Assume $q \equiv 7 \mod 16$. Then, as abelian groups, we have
$\Sel_{\fP^\infty} (B/J_\infty) = \BQ_2/\BZ_2$ and $\Sel_{\fP^{*\infty}} (B/J^*_\infty)  = \BQ_2/\BZ_2$.
\end{thm}
\noindent Note, however, that our proof of Theorem \ref{mir} provides no information about the action of Galois on the 
copies of $\BQ_2/\BZ_2$ appearing in Theorem \ref{f7}. We can also give an equivalent formulation of Theorem \ref{f7}
as follows. We have an exact sequence
\begin{equation}\label{f9}
0 \to B(J_\infty)\otimes_{\msb}\mst_{\fP}/\msb_{\fP} \to \Sel_{\fP^\infty} (B/J_\infty) \to \Sha(B/J_\infty)(\fP) \to 0,
\end{equation}
where  $\Sha(B/J_\infty)$ denotes the Tate-Shafarevich group of $B/J_\infty$, and $ \Sha(B/J_\infty)(\fP)$ is its $\fP$-primary subgroup.  There is also an entirely analogous sequence for 
$\Sel_{\fP^{*\infty}} (B/J^*_\infty)$. The following two remarks are valid for all primes $q \equiv 7 \mod 8$. We fix an algebraic closure $\ov{\BQ}$ of $\BQ$, and all of our relevant fields will be assumed to lie in it. Firstly, if $M$ is any algebraic extension of $K$, $B(M)\otimes_\BZ\BQ$ is a vector space over the field
$\mst$. Since $[\mst:\BQ] = 2h$, it follows that the $\BQ$-dimension of $B(M)\otimes_\BZ\BQ$, if it is finite, must  be a multiple of $2h$. Secondly, let $\sigma$ be any element
of $\Hom(J_\infty, \ov{\BQ})$ such that then restriction of $\sigma$ to $K$ is the non-trivial element of $\Gal(K/\BQ)$. Now $\sigma(J) = J$ because $J$ is Galois over $\BQ$,
and thus $\sigma(J_\infty)$ will be a $\BZ_2$-extension of $J$, which is unramified outside $\sigma(\fp) = \fp^*$. Hence $\sigma(J_\infty) = J_\infty^*$. Moreover, as the Gross curve $A$ is defined over the field $\BQ(j(\CO_K))$, the abelian variety $B$ is actually defined over $\BQ$. Hence applying $\sigma$ gives an isomorphism $B(J_\infty) \cong B(J^*_\infty)$.
Thus Theorem \ref{f7} has the following entirely equivalent formulation.

\begin{thm}\label{f10} Assume $q \equiv 7 \mod 16$. Then precisely one of the two following options is valid:- (i) both $B(J_\infty)\otimes_\BZ\BQ $ and $B(J^*_\infty)\otimes_\BZ\BQ $ have $\BQ$-dimension $2h$, and $\Sha(B/J_\infty)(\fP) = \Sha(B/J^*_\infty)(\fP^*) = 0$, or (ii) $B(J_\infty)\otimes_\BZ\BQ = B(J^*_\infty)\otimes_\BZ\BQ = 0$, and  both $\Sha(B/J_\infty)(\fP)$ and $\Sha(B/J^*_\infty)(\fP^*)$ are isomorphic to $\BQ_2/\BZ_2$ as abelian groups. 
\end{thm}

We next explain how we can strengthen this result a little by a purely arithmetic argument, and prove Theorem \ref{m3}.
As a first step, we note the following lemma (cf. Proposition 3.10 of \cite{CL}). Note that $[J_\infty:D_\infty] = 2$, and this extension is ramified precisely at the finite set of primes of $D_\infty$ lying above the rational prime $q$. In fact, for the primes $q \equiv 7 \mod 16$ there are precisely two primes of $D_\infty$ lying above $q$, because $\fq = \sqrt{-q}\CO_K$ is inert in $K_\infty/K$ by Lemma 3.3 of \cite{CL}, and therefore the unique prime of $K_1$ above $\fq$ must split in the extension $K_1(\sqrt{-1})/K_1$.

\begin{lem} \label{f11} For all primes $q \equiv 7 \mod 16$, the Galois group $\fD = \Gal(J_\infty/D_\infty)$ acts trivially on $\Sel_{\fP^\infty} (B/J_\infty)$.
\end{lem}
\begin{proof} Note first that $B_{\fP^\infty}$ is not contained in $D_\infty$ since the rational prime $q$ has ramification index 4 in the field $J_\infty$. Thus the non-trivial 
element $\fd$ of $\fD$ must act on $B_{\fP^\infty}$  by $-1$. Hence, if $f$ belongs to $\Hom(X(J_\infty), B_{\fP^\infty})$, we have $(\fd f)(x) = - f(\fd x)$, whence it follows that 
\begin{equation}\label{f13}
\Sel_{\fP^\infty} (B/J_\infty)^\fD  =  \Hom(X(J_\infty)/(1+\fd)X(J_\infty),  B_{\fP^\infty}).
\end{equation}
But $(1+\fd)X(J_\infty) \subset X(J_\infty)^\fD$, and we claim that $X(J_\infty)^\fD = 0$. Indeed, $X(J_\infty)_\fD$ is finite for all primes $q \equiv 7 \mod 16$ because
there are just $2$ primes of $D_\infty$ which ramify in $J_\infty$ and $X(D_\infty) = 0$ by Proposition \ref{prop: X(D)}. Hence, using the fact that $X(J_\infty) = \BZ_2$,
it follows that $X(J_\infty)^\fD = 0$, and the proof is complete.
\end{proof}

 Theorem \ref{m3} will now follow immediately from Theorem \ref{f10} and the following result. Let 
 \begin{equation}\label{f14}
 \theta_\infty : \Sel_{\fP^\infty} (B/D_\infty) \to \Sel_{\fP^\infty} (B/J_\infty)^\fD = \Sel_{\fP^\infty} (B/J_\infty)
\end{equation}
be the restriction map.

\begin{prop} The map $\theta_\infty$ is an isomorphism. \end{prop}

\begin{proof} We have the commutative diagram with exact rows
\[\xymatrix{
  0  \ar[r]^{} & \Sel_{\fP^\infty}(B/D_\infty) \ar[d]_{\theta_\infty} \ar[r]^{} & H^1(D_\infty, B_{\fP^\infty}) \ar[d]_{t_\infty} \ar[r]^{} & \prod_{v}H^1(D_{\infty,v},B)(\fP)\ar[d]_{m_\infty} \\
  0 \ar[r]^{} & \Sel_{\fP^\infty}(B/J_\infty)^{\fD} \ar[r]^{} & H^1(J_\infty, B_{\fP^\infty})^{\fD} \ar[r]^{} & \left(\prod_{w}H^1(J_{\infty, w}, B)(\fP)\right)^{\fD}.}\]
  Now $\fd$ must act on $B_{\fP^\infty}$ like $-1$ because $B_{\fP^2}$ does not belong to $B(D_\infty)$. Thus  $\fd + 1$ annihilates $B_{\fP^\infty}$, and $(\fd -1)B_{\fP^\infty} = B_{\fP^\infty}$, whence  $\Ker(t_\infty) = H^1(\fD, B_{\fP^\infty}) =0$, and so, by the snake lemma applied to the above diagram, we conclude that $\theta_\infty$ is injective.
 Moreover, $H^2(\fD, B_{\fP^\infty}) = (B_{\fP^\infty})^\fD = B_\fP$, and so
  $\Coker(t_\infty)$ is of order dividing 2, because it is a subgroup of $H^2(\fD, B_{\fP^\infty})$. Next we note that, for each place
  $v$ of $D_\infty$ which does not lie above $q$, $B$ has good reduction at $v$, and $v$ is unramified in the extension $J_\infty/D_\infty$, whence $H^1(\Gal(J_{\infty, w}/D_{\infty, v}), B(J_{\infty,w}))=0$ by a fundamental property of abelian varieties over local fields. Thus
  \begin{equation}\label{f15}
  \Ker(m_\infty) = \prod_{w|q}H^1(\fD, B(J_{\infty,w}))(\fP^\infty);
  \end{equation}
  here $w$ runs over the two primes of $J_\infty$ lying above $q$, which are, in fact, totally ramified in the extension $J_\infty/D_\infty$.  In particular, $\Ker(m_\infty)$ is annihilated by 2, and it then follows from the snake lemma applied to the above diagram that the cokernel of $\theta_\infty$ is annihilated by $8$, whence it must be zero because $\Sel_{\fP^\infty} (B/J_\infty) = \BQ_2/\BZ_2$. This completes the proof. 
\end{proof}

\bigskip

Up until now, all the results of this paper have been proven by classical arithmetic arguments based only on class field theory and Iwasawa theory. However, we see no way at present of ruling out the possibility that $B(D_\infty)\otimes_\BZ\BQ = 0$ in Theorem \ref{m3} without using the complex $L$-series attached to $B$ and Heegner points, inspired, of course, by the conjecture of Birch and Swinnerton-Dyer. Let $I_K$ be the idele group of $K$, and let
$$
\phi: I_K \to \mst ^\times
$$
be the Serre-Tate character  attached to $B/K$ (see \cite{ST}, Theorem 10). Let $\mst^+$ denote the maximal real subfield of $\mst$, so that $\mst^+$ has $h$ distinct embeddings   into the field $\BR$ of real numbers. Then each of these $h$ embeddings gives rise to a complex Grossencharacter $\phi_j$ of $K$, and the theory of complex multiplication shows that
$$
L(B/K, s) = \prod_{j=1}^{h} L(\phi_j, s)^2,
$$
where $L(\phi_j, s)$ denotes the Hecke $L$-function of $\phi_j$. By a theorem of Rohrlich \cite{R}, $L(B/K, 1) \neq 0$ for all primes $q \equiv 7 \mod 8$. Let $\chi$ be the quadratic character of $K$ corresponding to the quadratic extension $D/K$, and write
$L(\phi_j \chi, s)$ for the twist of $L(\phi_j,s)$ by $\chi$. By a root number calculation, it is shown in \cite{Gross1}, Theorem 19.1.1,  that $L(\phi_j \chi, s)$ has a zero at $s=1$ of odd multiplicity
for $j=1, \ldots, h$. 
\begin{thm}\label{simple}
For all primes $q \equiv 7 \mod 8$,  $L(\phi_j \chi, s)$ has a simple zero at $s=1$ for $j=1,...,h$,.
\end{thm}

\noindent In fact, this is a slightly stronger form of a special case of Theorem 2.2 of \cite{MY}, which asserts that, when $K =\BQ(\sqrt{-q})$, $L(\phi_j \chi, s)$ will have a simple zero
at $s=1$ for all sufficiently large primes $q \equiv 7 \mod 8$. Since 
$$
L(B/D, s) = \prod_{j=1}^{h} L(\phi_j \chi, s)^2L(B/K,s),
$$
Theorem \ref{mirB} then follows from Theorem \ref{simple} and the deep work of Gross-Zagier \cite{GZ} and Kolyvagin-Logachev \cite{KL} applied to the twist of $B/K$ by the quadratic character of the extension $D/K$. 

\medskip

We now sketch the proof of Theorem \ref{simple}. As we follow closely the proof given in \cite{MY}, we shall largely just explain those details needed to check that the arguments in \cite{MY} do indeed remain valid for all primes $q \equiv 7 \mod 8$. For brevity, put $\psi_j = \phi_j\chi$, so that $\psi_j$ is a Grossencharacter of $K$ with conductor $4\fq$, where, as always, $\fq = \sqrt{-q}\CO_K$, and we know by \cite{Gross1}, Theorem 19.1.1 that $L(\psi_j, s)$ has a zero of odd multiplicity at $s=1$ for all $j \in \{1, \ldots, h\}$. Moreover,  it is known that the set $\{\psi_1,..., \psi_h\}$ coincides with the set $\{\psi \omega\}$, where we have fixed one embedding of $\mst$ in $\BC$, and written $\psi = \psi_1$ for the corresponding Grossencharacter, and $\omega$ then runs over the set of all complex characters of the ideal class group of $K$. For each ideal class $C$ of $K$, we define the partial $L$-series $L(\psi, C, s) = \sum_{\fa \in C}\psi (\fa)(Na)^{-s}$, where the sum is taken over all integral ideals in $C$, which are prime to $4\fq$. Thus, for every complex character $\omega$ of the ideal class group of $K$, we have
$L(\psi \omega, s) = \sum_{C} \omega(C) L(\psi, C, s)$, where the sum is taken over all ideals classes $C$ of $K$. Now the Gross-Zagier formula (see \cite{GZ}, Corollary 1.3) shows that $L(\psi, s)$ will have a simple zero at $s=1$
if and only if $L(\psi_j, s)$ has a simple zero at $s=1$ for all $j \in \{1, \ldots, h\}$, or equivalently $L(\psi\omega,s)$ has a simple zero at $s=1$ for all complex characters $\omega$
of the ideal class group of $K$. But, for every ideal class $C$ of $K$, we have the identity
\begin{equation}\label{f16}
hL'(\psi, C, s) = \sum_\omega \ov{\omega}(C)L'(\psi \omega, s),
\end{equation}
 where the sum is taken over all characters $\omega$ of the ideal class group of $K$. Thus, if the partial $L$-series $L(\psi, C, s)$ has a simple zero at $s=1$ for some ideal class $C$ of $K$, it follows from this identity and the Gross-Zagier formula that $L(\psi \omega, s)$ will have a simple zero at $s=1$
 for all characters $\omega$ of the ideal class group of $K$. We now proceed to show by analytic arguments  that, in particular, $L(\psi, \fC, S)$ has a simple zero at $s=1$, where $\fC$ denotes the class of principal ideals of $K$. 
 
 For every $\beta \in \CO_K$, with $(\beta, 4\fq) = 1$, we have $\psi(\beta\CO_K) = \epsilon(\beta)\beta$, where $\epsilon$ is a quadratic character of $(\CO_K/4\fq)^\times$. Put $\Lambda(\psi, \fC, s) = (W/2\pi)^s\Gamma(s)L(\psi, \fC, s)$, where $\Gamma(s)$ denotes the classical $\Gamma$-function, and $W= 4q$.  Define $f(z) = z^{-1}\int_{z}^{\infty} e^{-t}t^{-1}dt$ for any strictly positive real number $z$. By formula $(2.11)$ of \cite{MY}, which the authors attribute to Rohrlich, we have
\begin{equation}\label{f17}
2^{-1}\Lambda'(\psi, \fC, 1) = U + V,
\end{equation}
where
\begin{equation}\label{f18}
U = \sum_{n \geq 1, (n, W) = 1}\epsilon(n)n f(2\pi n^2/W), 
\end{equation}
and
\begin{equation}\label{f19}
V = \sum_{n \geq 1, (n, W) = 1}a_nf(2\pi n/W), \, \, \, a_n = 2\sum_{x,y}\epsilon(x + y\sqrt{-q})x \, \, \, ;
\end{equation}
here the second sum is taken over all strictly positive integers $x$ and $y$ such that $x^2 + qy^2 = n$. Thus to complete the proof of Theorem \ref{simple}, we must show that
$U > |V|$ for all primes $q \equiv 7 \mod 8$.

\medskip

It follows from Proposition 3.1 and Corollary 3.4 of \cite{MY} that
\begin{equation}\label{f20}
U > W(0.5235 - 0.8458W^{-1/4} - 0.3951W^{-1/2}).
\end{equation}
But $W = 4q \geq 28$, so that, substituting this lower bound for W in the two negative fractional powers of $W$ in \eqref{f20}, we immediately obtain
\begin{equation}\label{f21}
U > 0.08107226W.
\end{equation}
On the other hand, we have
\begin{equation}\label{f22}
|V| \leq \sum_{x, y} 2xf(\pi(x^2 + qy^2)/(2q)),
\end{equation}
where the sum on the right is now taken over all strictly positive integers $x$ and $y$. Note however that $f(z) <  z^{-2}e^{-z}$ for all strictly positive real numbers $z$, whence
\begin{equation}\label{f23}
|V| \leq \sum_{x,y}2xe^{-\frac{\pi x^2}{2q}} e^{-\frac{\pi y^2}{2}}(\pi/2(x^2 + y^2/q))^{-2} < \frac{8}{\pi^2} (\sum_{x=1}^{\infty}xe^{-\frac{\pi x^2}{2q}})(\sum_{y=1}^{\infty}y^{-4}e^{-\frac{\pi y^2}{2}}).
\end{equation}
where again the first sum is over all strictly positive integers $x$ and $y$. By Lemma 4.2 of \cite{MY}, we have $\sum_{x=1}^{\infty}xe^{-\frac{\pi x^2}{2q}} < q/\pi$, and by a direct numerical calculation one has $\sum_{y=1}^{\infty}y^{-4}e^{-\frac{\pi y^2}{2}} < 0.2080$. Thus we obtain
\begin{equation}\label{f24}
|V| < \frac{2}{\pi^3}0.2080 W = 0.01341663832 W  < U.
\end{equation}
This completes the proof.

\appendix
\section{}\label{sec: appendix}
 Unlike the earlier part of this paper,  $p$ will now denote an arbitrary prime number throughout this Appendix. Also $K$ will now denote either the rational field $\BQ$, or an arbitrary imaginary quadratic field. Let $\fp$ be a prime of $K$ such that the completion $K_\fp$ is $\BQ_p$. We write $F$ now for an arbitrary finite extension of $K$. However, if $K=\BQ$, we shall always assume in addition that $F$ is totally real. By class field theory, $K$ admits a unique $\BZ_p$-extension which is unramified outside $\fp$, and we denote this $\BZ_p$-extension by $K_\infty$. Let $K_n$ be its $n$-th layer. Write $\mathscr{S}$ for the set of primes of $F$ lying above $\fp$. 
Let $M$ be the maximal abelian $p$-extension of $F$, which is unramified outside $\mathscr{S}$. The following classical formula is due to the first author \cite{Coa77} when $K=\BQ$, and the first author and Wiles \cite{CW} when $K$ is imaginary quadratic. Let $R_\fp$ be the $\fp$-adic regulator of $F$, whose definition is recalled below \eqref{eq: appendix1}. Let $F_\infty = FK_\infty$. 

\begin{thm}\label{cw}
	With notation as above, the degree $[M:F_\infty]$ is finite if and only if the $\fp$-adic regulator $R_\fp$ of $F$ is non-zero. If $R_\fp\neq 0$, then
	\begin{equation}
	 [M:F_\infty]= \frac{2p^{e-f+1}R_\fp h_F}{\omega_F \sqrt{\Delta_\fp}}\prod_{\mathfrak{g}\in \mathscr{S}}(1-(N\mathfrak{g})^{-1}),\quad  \text{up to multiplication by a }p\text{-adic unit}.
	\end{equation}
	Here $h_F$ is the class number of $F$,  $\omega_F$ is the number of roots of unity in $F$, $\Delta_\fp\in \CO_{K_\fp}$ is a generator of the $\fp$-component ideal of the relative discriminant of $F/K$,  and $N(\mathfrak{g})$ is the absolute norm of an ideal $\mathfrak{g}$ of $F$. Moreover, the integers $e$ and $f$ are defined by $F\cap K_\infty = K_e$ and $H\cap K_\infty = K_f$, respectively, where $H$ denotes the Hilbert class field of $K$ (thus $f=0$ when $K = \BQ$).
\end{thm}

\noindent This formula is stated and proved in \cite{Coa77} and \cite{CW} under the assumptions that the prime $p$ is odd, and that $p$ does not divide the class number of $K$. However, it is often precisely these cases which are needed for certain applications of the formula, as, for example, in the present paper.
Thus the aim of this appendix is to show that, after some very slight modifications, the arguments given in the two original papers work in complete generality, and give a proof of Theorem~\ref{cw}. All the notation we used in this appendix is compatible with \cite{CW}.

\medskip

Let $d=[F:K]$. We first recall the definition of the $\fp$-adic regulator $R_\fp$. Let $\varepsilon_1,\cdots, \varepsilon_{d-1}$ be a basis of the group $\mathcal{E}$ of global units of $F$ modulo torsion, and put $\varepsilon_d=1+p$.
Let $\ov{K}_\fp$ be a fixed algebraic closure of $K_\fp(=\BQ_p)$. Denote by $\phi_1,\cdots, \phi_d$ the distinct embeddings of $F$ into $\ov{K}_\fp$.   Let $\log$ denote the usual extension of the  $p$-adic logarithmic function to $\ov{K}_\fp$. The $\fp$-adic regulator $R_\fp$ is then defined to be the $d\times d$ determinant 
\begin{equation}\label{eq: appendix1} R_\fp = (d\log \varepsilon_d)^{-1} \det (\log(\phi_i(\varepsilon_j)  ) )_{1\leq i, j\leq d}. 
\end{equation}
In fact, since $\sum_{i=1}^{d} \log(\phi_i(\varepsilon_j) )= 0$ for $j=1, \cdots, d-1$, it is not difficult to see  that $R_\fp = \det(\log(\phi_i(\varepsilon_j)))_{1\leq i, j\leq d-1}.$
Let $V=1+p\BZ_p$ be the group of principal units of $K_\fp$, and define $\bar{V} =V\{ \pm 1 \}/\{ \pm 1 \}$. Note that the $p$-adic logarithm defines an isomorphism $\bar{V} \cong \BZ_p$ for all primes $p$. Our arguments will hinge on the following elementary facts. Firstly, we have that $\log(\varepsilon_d)=\log(1+p)=2p$, up to a $p$-adic unit.
Secondly, for an integer $n=p^u m$ with $p\nmid m$, the image of the closure $\langle (1+p)^n \rangle \subset V$ under the map $V \rightarrow \bar{V}$ has index $p^u$ in $\bar{V}$.  

\medskip

Since $K$ is either $\BQ$ or an arbitrary imaginary quadratic field,  we can now only assert that there will be only finitely many primes of $K_\infty$ lying above $\fp$. We fix such one prime $\tilde{\fp}$ of $K_\infty$ lying above $\fp$. Let $\varPsi_n$ be the completion of $K_n$ at the unique prime below $\tilde{\fp}$. Let $\varPsi_\infty= \cup_{n} \varPsi_n$, so that $\Gal(\Psi_\infty/ \Psi_0) \cong \BZ_p$. 
Now our hypothesis $H\cap K_\infty = K_f$ clearly implies that $\varPsi_\infty/\varPsi_f$ is totally ramified.
Let $V_n$ denote the group of principal units of $\varPsi_n$ which are $\equiv 1$ modulo the maximal ideal. Thus $V_0 = V$.  Let $N_n$ denote the norm map from $V_n$ to $V$, and define $\bar{N}_n$ to be the composite map 
\[\bar{N}_n: V_n \xrightarrow{N_n} V \rightarrow \bar{V}.  \] 
Of course, $\bar{N}_n=N_n$ for $p>2$.
\begin{lem}\label{lem: filtration}
	For each $n\geq f$, we have $\bar{N}_n(V_n)= \bar{V}^{p^{n-f}} $.
\end{lem}

\begin{proof}
Recall that the inertial subgroup of $\fp$ in the extension $K_\infty/K$ is $\Gal(K_\infty/K_f)$.  It follows then from local class field theory that, assuming $n \geq f$, the Artin map induces two isomorphisms
	\[V/\bigcap_{n\geq f} N_n(V_n) \cong G(\varPsi_\infty/\varPsi_f)  \quad \text{ and }\quad  V/{N_n(V_n)}\cong G(\varPsi_n/\varPsi_f).\]
Since $G(\varPsi_\infty/\varPsi_f)\cong \BZ_p$,	it follows that we must have $-1 \in \bigcap_{n\geq f} N_n(V_n)$, whence
	\[  \bar{V}/{\bar{N}_n(V_n)} \cong V/N_n(V_n) \cong \BZ/{p^{n-f}\BZ}, \, (n \geq f).\]
But the group $\bar{V}$ is isomorphic to $\BZ_p$, and so
 $\bar{N}_n(V_n)$, being a closed subgroup of index $p^{n-f}$, must be equal to $\bar{V}^{p^{n-f}},$  and the proof of the lemma is complete.
\end{proof}

Define $U_1=\prod_{\mathfrak{g}\in \mathscr{S}}{ U_{\mathfrak{g},1}}$, where $U_{\mathfrak{g},1}$ denotes the local units $\equiv 1 \mod \fg$ in the completion of $F$ at $\fg$, and let $N_{F/K}$ be the norm map from $U_1$ to $V$. Define $\bar{N}_{F/K}$ to be the composite map
\[ \bar{N}_{F/K}: U_1 \xrightarrow{N_{F/K}} V \rightarrow \bar{V}.\]
 The following two lemmas describe the kernel and image of $\bar{N}_{F/K}$. For each $n \geq 0$, let $F_n$ be the $n$-th layer of the $\BZ_p$-extension $F_\infty/F$, and let $C_n$ be the idele class group of $F_n$. Put
 \begin{equation}
 Y = \cap_{n \geq 0}N_{F_n/F}C_n.
 \end{equation}

\begin{lem}\label{lem: kernel}
	$Y\cap U_1$ is the kernel of $\bar{N}_{F/K}$.
\end{lem}

\begin{proof}
	This is equivalent to showing that $Y\cap U_1$ is the inverse image of $\pm 1$ under the map $N_{F/K}$. By Lemma~\ref{lem: filtration}, $\bigcap_n N_n(V_n)=\{\pm 1  \}$ for $p=2$ whence $\bigcap \bar{N}_n(V_n)=1$. Granted this, the rest of the proof is exactly the same as that given in \cite[Lemma 5]{CW}, and we omit the details.
\end{proof}
\begin{lem}\label{lem: image}
	Let $L$ be the $p$-Hilbert class field of $F$ and let the integer $k$ be defined by $L\cap F_\infty=F_k$. Then $\bar{N}_{F/K}(U_1)=\bar{V}^{p^{e+k-f}}$.
\end{lem}
\begin{proof}
The proof is essentially the same as the one of \cite[Lemma 6]{CW}. Firstly, one needs to replace $N_{F/K}$ and $N_n$ there by $\bar{N}_{F/K}$ and $\bar{N}_n$, respectively. 
 Clearly, $k+e-f\geq 0$. By Lemma~\ref{lem: kernel}, $\bar{N}_{k+e}(V_{k+e})=\bar{V}^{p^{k+e-f}}$. Thus one also needs to replace the integer $e$ in \cite{CW} by $e-f$, and $t$ by $t+f$ consequently. The rest of the argument is the same as in \cite[Lemma 6]{CW}.
\end{proof}

Let $\mathcal{E}_1$ be the group of global units of $F$, which are $\equiv 1\bmod \mathfrak{g}$ for each $\mathfrak{g}\in \mathscr{S}$. 
Let $j: F\rightarrow \prod_{\mathfrak{g}\in \mathscr{S}}F_\mathfrak{g}$ be the canonical embedding. Define $D$ to be the $\BZ_p$-submodule of $U_1$ which is generated by $j(\mathcal{E}_1)$ and $j(\varepsilon_d)$.

\begin{lem}\label{lem: index1}
The index of $D$ in $U_1$ is finite if and only if $R_\fp\neq 0$. If $R_\fp\neq 0$, then
	\[ [U_1:D] = \frac{2pd}{\omega_F \sqrt{\Delta_\fp}} \prod_{\mathfrak{g}\in \mathscr{S}}(1-(N\mathfrak{g})^{-1}) \quad  \text{up to a }p\text{-adic unit}. \]
\end{lem}
\begin{proof}
	This is proven in \cite[Lemma 7,8 and 9]{CW}. The proof there applies to our case without change, but we remind the reader that, when $p=2$,  $\log(\varepsilon_d)=\log(3)=2p$, up to a $2$-adic unit. This gives the extra factor $2$ in Lemma~\ref{lem: index1} when comparing with \cite[Lemma 9]{CW}.
\end{proof}

We now come to the crucial lemma. 
If $p>2$, we clearly have $N_{F/K}(j(\mathcal{E}_1))=\{1\}$,  since $N_{F/K}(j(\mathcal{E}_1))$ is contained in $1+p\BZ_p$. If $p=2$, then the unit group of $K$ is $\{\pm 1 \}$, since $2$ does not split in $\BQ(i)$ or $\BQ(\sqrt{-3})$. Thus $N_{F/K}(\overline{j(\mathcal{E}_1)})\subset \{\pm 1\}$. In other words, $\overline{j(\mathcal{E}_1)}$ is contained in the kernel of the map $\bar{N}_{F/K}$. Thus Lemma~\ref{lem: kernel} shows that $\overline{j(\mathcal{E}_1)}$ is contained in $Y\cap U_1$. 

\begin{lem}\label{lem: index2}
Let the integer $k$ be as in Lemma~\ref{lem: image}. The index of $\overline{j(\mathcal{E}_1)}$ in $Y\cap U_1$ is finite if and only if $R_\fp\neq 0$. If $R_\fp\neq 0$, then
	\[ [Y\cap U_1:\overline{j(\mathcal{E}_1)}]= \frac{2p^{e+k-f+1}R_\fp}{\omega_F \sqrt{\Delta_\fp}}\prod_{\mathfrak{g}\in \mathscr{S}}(1-(N\mathfrak{g})^{-1})\quad  \text{up to a }p\text{-adic unit}.  \]
\end{lem}

\begin{proof}
	We have the commutative diagram with exact rows
	\[ \begin{tikzcd}
	0 \ar[r] & Y\cap U_1 \ar[r] & U_1 \ar[r, "\bar{N}_{F/K}"] &  \bar{V}^{p^{e+k-f}}\ar[r] & 0\\
	0 \ar[r]  & \overline{j(\mathcal{E}_1)} \ar[r] \ar[u,hook] & D \ar[r, "\bar{N}_{F/K}"] \ar[u,hook]& \bar{V}^d \ar[r] \ar[u,hook]& 0  
	\end{tikzcd}  \]
	The exactness of the first row follows from Lemma~\ref{lem: kernel} and Lemma~\ref{lem: image}. Note that $\bar{N}_{F/K}(D)$ is the closure of the image of $\langle (1+p)^d \rangle $ in $\bar{V}$, which coincides with $\bar{V}^d$. Since $\bar{V}\cong \BZ_p$, we have that $[\bar{V}^{p^{e+k-f}}: \bar{V}^d]=d/p^{e+k-f} $, up to a $p$-adic unit. Thus Lemma~\ref{lem: index2} follows from Lemma~\ref{lem: index1}.
\end{proof}

\begin{proof}[Proof of Theorem~\ref{cw}]
	As in \cite[Theorem 11]{CW},  noting that we have already shown that $\overline{j(\mathcal{E}_1)}$ is contained in $ Y\cap U_1$,
	it is a standard consequence of global class field theory that
	\[(Y\cap U_1)/\overline{j(\mathcal{E}_1)}  \cong G(M/LF_\infty),  \]
	where, as in Lemma~\ref{lem: image}, $L$ is the $p$-Hilbert class field of $F$. It follows that
	\[ [M:F_\infty]=[M:LF_\infty][LF_\infty:F_\infty] =[Y\cap U_1:\overline{j(\mathcal{E}_1)} ]h_F/p^k   \quad  \text{up to a }p\text{-adic unit}. \]
	The last equality follows from $[LF_\infty:F_\infty]=[L:L\cap F_\infty]$ and the definition of $k$ given in Lemma~\ref{lem: image}.
	Now, Theorem~\ref{cw} follows from Lemma~\ref{lem: index2}.
\end{proof}

\bigskip

We end with the following remark. Let $F$ be a totally real number field of degree $d$. The reason for proving Theorem \ref{cw} in 1974 was that it provided the first general evidence
that Iwasawa's then revolutionary discovery of his $p$-adic  "main conjecture" for totally real $F$ lying inside the field generated by all $p$-power roots of unity might hold in complete generality for the cyclotomic $\BZ_p$-extension of all totally real number fields $F$. The deep subsequent work of Mazur-Wiles and Wiles has happily shown this to be true for all totally real $F$ and all odd primes $p$. However, the situation for the prime $p=2$ still has not been completely settled, and we want to just point out that, even in this special case, Theorem \ref{cw}
is in perfect accord with the "main conjecture" which we believe to be true.
Let $\mathfrak{M}$ be the maximal abelian $p$-extension of $F$, which is unramified outside $p$ and the infinite primes of $F$. Let $F_\infty$ be the cyclotomic $\BZ_p$-extension of $F$. Let $I_F$ be the id\`{e}le group of $F$.
It follows from class field theory that we have the following commutative diagram; all the maps are clearly surjective and the kernels of the maps on the rows are finite groups of order prime to $p$:
\[ 
\begin{tikzcd}
I_F/\overline{F^\times \prod\limits_{v \nmid p \infty}U_v  \prod\limits_{v\mid \infty}\BR_{>0} } \ar[r]\ar[d] & \Gal(\mathfrak{M}/F) \ar[d] \\
I_F/ \overline{F^\times \prod\limits_{v \nmid p\infty }U_v \prod\limits_{v\mid \infty}\BR^\times} \ar[r]             &       \Gal(M/F).       \\
\end{tikzcd}
\]

\begin{lem}
	The kernel of the vertical map on the left is isomorphic to $(\BR^\times/\BR_{>0})^d$.
\end{lem}
\begin{proof}
Note that $\prod_{v\mid \infty}\BR^\times $ naturally maps onto this kernel. Thus, the assertion will directly follow from the following identity in $I_F$:
\begin{equation}\label{eq: appendix2}
	\left(\prod_{v\mid \infty}\BR^\times \right)  \bigcap \left(\overline{ F^\times  \prod_{v\nmid p\infty}U_v \prod_{v\mid \infty}\BR_{>0}}\right) = \prod\limits_{v\mid \infty} \BR_{>0}.
\end{equation}	
	To see this identity, note that by definition an element $(a_v)_v$ of  $\prod_{v\mid \infty}{\BR^\times}\subset I_K$ has component $1$ at every finite place.
	 Let $y$ be an element in the second group of \eqref{eq: appendix2}. Then $y$ is a limit of  $x^{(n)} \cdot b^{(n)}$ where $x^{(n)}\in F^\times$ and $b^{(n)} =(b^{(n)}_v)_v\in \prod_{v\nmid p\infty} \prod_{v\mid \infty}\BR_{>0}$ for each $n$. But note that $b^{(n)}_v=1$ if $v\mid p$. This forces $x^{(n)}=1$ for each $n$ whence $b^{(n)}_v>0$ for $v\mid \infty$ and $n$ is large. This proves \eqref{eq: appendix2}, completing the proof of the lemma.
\end{proof}
Thus, thanks to this lemma, we obtain the following result from Theorem~\ref{cw}.
\begin{thm}\label{tr}
	Let $F$ be a totally real number field of degree $d$. Let $F_\infty$ be the cyclotomic $\BZ_p$-extension of $F$, and let $\mathfrak{M}$ be the maximal abelian $p$-extension of $F$ which is unramified outside $p$ and the $infinite$ primes. 
	Let $e$ and $h_F$ as defined in Theorem~\ref{cw}. Let $\Delta$ be the discriminant of $F$. Assuming the $p$-adic regulator $R_p$ of $F$ is nonzero, we have 
	\[  [\mathfrak{M}:F_\infty]=  \frac{2^{d}p^{e+1}R_p h_F}{ \sqrt{\Delta}}\prod_{v\mid p}(1-(Nv)^{-1})\quad, \text{up to a }p\text{-adic unit}.   \]
\end{thm}

Let $\zeta_{F,p}(s)$ be the $p$-adic zeta function of $F$, constructed by P. Cassoun-Nogues and P. Deligne and K. Ribet. Assuming that $R_p \neq 0$, Colmez \cite{Col} proved that the residue of $\zeta_{F,p}$ at $s=1$ is 
\[
\frac{2^{d-1}R_p h_F}{\sqrt{\Delta}} \prod_{v\mid p}(1-(Nv)^{-1}).
\]
We end by pointing out that Colmez's formula and Theorem \ref{tr} are in perfect accord for all primes $p$, including $p=2$, via the following "main conjecture" of Iwasawa theory.
Let $\mathfrak{M}_\infty$ be the maximal abelian $p$-extension of $F_\infty$ which is unramified outside $\fp$ and the infinite primes, and put $\fX_\infty = \Gal(\mathfrak{M}_\infty/F_\infty)$.  Let $\Gamma = \Gal(F_\infty/F)$, and let $\Lambda(\Gamma)$ be the Iwasawa algebra of $\Gamma$. Now, fixing a topological generator $\gamma$ of $\Gamma$, we can identify $\Lambda(\Gamma)$ with the ring of formal power series $\BZ_p[[T]]$ by mapping $\gamma$ to $1+T$. Let $\kappa: \Gamma \to \BZ_p^\times$ be the cyclotomic character of $\Gamma$, and put $u = \kappa(\gamma)$. Thus, by the very definition of the cyclotomic character,
we have $u - 1$ is equal to $2p^{e+1}$ up to a $p$-adic unit.
Now Iwasawa \cite{Iwa2} has shown that $\fX_\infty$ is a finitely generated torsion $\Lambda(\Gamma)$-module with no non-zero finite $\Lambda(\Gamma)$-submodule. Thus, by the structure theory of such modules, we can associate to $\fX_\infty$ a characteristic power series $f_{\fX_\infty}(T)$ in $\BZ_p[[T]]$. Moreover, $(\fX_\infty)_\Gamma = \Gal(\mathfrak{M}/F_\infty)$. It then follows from the Euler characteristic formula that, assuming $R_p \neq 0$, we have
\begin{equation}\label{e1}
[\mathfrak{M}:F_\infty] = f_{\fX_\infty}(0), \, \,  \text{up to a }p\text{-adic unit}.
\end{equation}
Now the "main conjecture" for $\fX_\infty$ asserts that, for a suitable choice of the characteristic power series $f_{\fX_\infty}(T)$, we have 
\begin{equation}\label{e2}
\zeta_{F,p}(s) = f_{\fX_\infty}(u^{s-1} - 1)/(u^s - u),
\end{equation}
where, as above, $\zeta_{F,p}(s)$ is the $p$-adic zeta function of $F$. Also
\begin{equation}\label{e3} 
u^{s-1} - 1 = (s-1)\log(u)\,\,   + \, \, \text{ higher powers of }(s-1),
\end{equation}
and $\log(u) = 2p^{e+1}$, up to a $p$-adic unit. 
If we now combine Theorem \ref{tr} with \eqref{e1}, \eqref{e2}, \eqref{e3}, we see that, when $R_p \neq 0$,  the "main conjecture" does indeed predict Colmez's residue formula up to a $p$-adic unit for all primes $p$, as claimed above.

\medskip

\medskip

\noindent John Coates,\\
Emmanuel College, Cambridge,\\
England.\\
{\it jhc13@dpmms.cam.ac.uk }

\medskip

\noindent Jianing Li \\
CAS Wu Wen-Tsun Key Laboratory of Mathematics,   University of Science and Technology of China, \\
Hefei, Anhui 230026, China. \\
{\it lijn@ustc.edu.cn}

\medskip

\noindent Yongxiong Li,\\
Yau Mathematical Sciences Center,\\
Tsinghua University, \\
Beijing, China.\\
{\it liyx\_1029@mail.tsinghua.edu.cn}

\end{document}